%% file: main.tex
\documentclass{amsart}
  
\usepackage{amssymb,amsfonts}
\usepackage[margin=1.5in]{geometry}
\usepackage{mathtools}
\usepackage{color}
\usepackage[color,matrix,all,arc]{xy}
\usepackage[shortlabels]{enumitem}
\usepackage{mathrsfs}
\usepackage{eucal}
\usepackage{graphicx}
\usepackage{float}
\usepackage{url} 
\usepackage{bbm,mathscinet}
\usepackage[dvipsnames]{xcolor}
\usepackage{soul}
\definecolor{allrefcolors}{rgb}{.05,.45,.6}
\usepackage[pagebackref,linktocpage=true,colorlinks=true,allcolors=allrefcolors,bookmarksopen,bookmarksdepth=3]{hyperref}
\usepackage[nameinlink,noabbrev,capitalise]{cleveref}
\usepackage{mathtools}
\usepackage{subcaption}

\usepackage{tikz,tikz-cd}
\usepackage{pgfplots}

\usepackage{enumitem}
\setenumerate{label=(\roman*),topsep=1pt,itemsep=1pt,partopsep=0pt,parsep=0pt}
\pgfplotsset{compat=1.16}

\newcommand{\ang}[1]{\langle #1 \rangle}

\newtheorem{thm}{Theorem}[section]

\newtheorem{lem}[thm]{Lemma}
\newtheorem{cor}[thm]{Corollary}
\newtheorem{por}[thm]{Porism}
\newtheorem{quest}[thm]{Question}

\theoremstyle{definition}
\newtheorem{defn}[thm]{Definition}
\newtheorem{exmp}[thm]{Example}
\newtheorem{notn}[thm]{Notation}
\newtheorem{asmpt}[thm]{Assumption}

\theoremstyle{remark}
\newtheorem{rem}[thm]{Remark}

\crefname{thm}{Theorem}{Theorems}
\crefname{cor}{Corollary}{Corollaries}
\crefname{por}{Porism}{Porisms}
\crefname{defn}{Definition}{Definitions}
\crefname{lem}{Lemma}{Lemmas}
\crefname{prop}{Proposition}{Propositions}
\crefname{asmpt}{Assumption}{Assumptions}
\crefname{quest}{Question}{Questions}
\crefname{rem}{Remark}{Remarks}
\crefname{exmp}{Example}{Examples}

\AddToHook{env/prop/begin}{\crefalias{thm}{prop}}
\AddToHook{env/lem/begin}{\crefalias{thm}{lem}}
\AddToHook{env/cor/begin}{\crefalias{thm}{cor}}
\AddToHook{env/defn/begin}{\crefalias{thm}{defn}}
\AddToHook{env/exmp/begin}{\crefalias{thm}{exmp}}
\AddToHook{env/notn/begin}{\crefalias{thm}{notn}}
\AddToHook{env/asmpt/begin}{\crefalias{thm}{asmpt}}
\AddToHook{env/rem/begin}{\crefalias{thm}{rem}}
\AddToHook{env/quest/begin}{\crefalias{thm}{quest}}
\AddToHook{env/por/begin}{\crefalias{thm}{por}}

\numberwithin{equation}{section}

\newcommand{\IC}{\mathbb{C}}

\newcommand{\IF}{\mathbb{F}}

\newcommand{\IR}{\mathbb{R}}
\newcommand{\IS}{\mathbb{S}}

\newcommand{\IZ}{\mathbb{Z}}

\newcommand{\sA}{\mathcal{A}}

\newcommand{\sC}{\mathcal{C}}
\newcommand{\sD}{\mathcal{D}}
\newcommand{\sE}{\mathcal{E}}
\newcommand{\sF}{\mathcal{F}}
\newcommand{\sG}{\mathcal{G}}
\newcommand{\sH}{\mathcal{H}}

\newcommand{\sJ}{\mathcal{J}}

\newcommand{\sL}{\mathcal{L}}
\newcommand{\sM}{\mathcal{M}}
\newcommand{\sN}{\mathcal{N}}
\newcommand{\sO}{\mathcal{O}}
\newcommand{\sP}{\mathcal{P}}

\newcommand{\sR}{\mathcal{R}}
\newcommand{\sS}{\mathcal{S}}

\newcommand{\sU}{\mathcal{U}}
\newcommand{\sV}{\mathcal{V}}
\newcommand{\sW}{\mathcal{W}}
\newcommand{\sX}{\mathcal{X}}
\newcommand{\sY}{\mathcal{Y}}

\newcommand{\fm}{\mathfrak{m}}

\newcommand{\fo}{\mathfrak{o}}

\renewcommand{\mod}[1]{#1{\operatorname{\mathbf{-mod}}}}
\newcommand{\hooklongrightarrow}{\lhook\joinrel\longrightarrow}
\newcommand{\osr}{\overline{\mathcal{R}}}
\newcommand{\CP}{{\mathbb{C}\mathrm{P}}}
\newcommand{\OP}{{\mathbb{O}\mathrm{P}}}
\newcommand{\RP}{{\mathbb{R}\mathrm{P}}}
\newcommand{\Pic}{\mathrm{Pic}}
\newcommand{\ev}{\mathrm{ev}}
\newcommand{\id}{\mathbbm{1}}

\DeclareMathOperator{\Core}{Core}
\DeclareMathOperator{\Ho}{Ho}
\DeclareMathOperator*{\hocolim}{hocolim}
\DeclareMathOperator{\ind}{ind}
\DeclareMathOperator{\Hw}{Hw}
\DeclareMathOperator{\Nbhd}{Nbhd}
\DeclareMathOperator{\Ob}{Ob}
\DeclareMathOperator{\Ends}{Ends}
\DeclareMathOperator{\Map}{Map}
\DeclareMathOperator{\End}{End}
\DeclareMathOperator{\Hom}{Hom}
\DeclareMathOperator*{\holim}{holim}
\DeclareMathOperator{\hofib}{hofib}

\DeclareMathOperator{\Gr}{Gr}
\DeclareMathOperator{\Perf}{Perf}

\begin{document}

\title{Homotopy rigidity of nearby Lagrangian cocores}

\author{Johan Asplund}
\address{Department of Mathematics, Stony Brook University, 100 Nicolls Road, Stony Brook, NY 11794}
\email{johan.asplund@stonybrook.edu}
\author{Yash Deshmukh}
\address{School of Mathematics, Institute for Advanced Study, 1 Einstein Drive, Princeton, NJ 08540, USA}
\email{deshmukh@ias.edu}
\author{Alex Pieloch}

\begin{abstract}
An exact Lagrangian submanifold $L \subset X^{2n}$ in a Weinstein sector is called a nearby Lagrangian cocore if it avoids all Lagrangian cocores and is equal to a shifted Lagrangian cocore at infinity. Let $k$ be the dimension of the core of the subcritical part of $X$. For $n \geq 2k+2$ we prove that that the inclusion of $L$ followed by the retract to the Lagrangian core of $X$ and the quotient by the $(n-k-1)$-skeleton of the core, is null-homotopic. As a consequence, in many examples, a nearby Lagrangian cocore is smoothly isotopic (rel boundary) to a Lagrangian cocore in the complement of the missed Lagrangian cocores. The proof uses the spectral wrapped Donaldson--Fukaya category with coefficients in the ring spectrum representing the bordism group of higher connective covers of the orthogonal group.
\end{abstract}
\maketitle
\tableofcontents

\input{Introduction.tex}

\subsection*{Acknowledgments}
We thank Georgios Dimitroglou Rizell for helpful discussions, and JA thanks Ceyhun Elmacioglu for helpful conversations. We thank Laurent Côté, Georgios Dimitroglou Rizell, Noah Porcelli, and Ivan Smith for comments on a draft of this paper.

Part of this work was completed at a Summer Collaborators Program hosted by the Institute for Advanced Study; we are grateful for their hospitality and financial support. JA was partially supported by the Knut and Alice Wallenberg Foundation and the Swedish Royal Academy of Sciences. YD thanks Max Planck Institute for Mathematics in Bonn and Institute for Advanced Study for their hospitality and financial support. YD was supported by the National Science Foundation
grant DMS-2424441. AP was supported by a National Science Foundation Postdoctoral Research Fellowship through NSF grant DMS-2202941.

\input{Rbranes.tex}
\input{SpectralFukaya.tex}
\input{OC.tex}

\input{NearbyCocoreEquivalent.tex}
\input{Proof.tex}

\appendix
\input{Fold.tex}
\input{DetectingNullhomotopy.tex}
\input{GeometricBackground.tex}

\bibliographystyle{alpha}
\bibliography{references.bib}

\end{document}

%% file: Introduction.tex
\section{Introduction}\label{sec:intro}

\subsection{Statement of results}
Let $X^{2n}$ be a Weinstein sector with a fixed Weinstein handlebody decomposition (see \cref{dfn:handlebody_decomp_sector}). This induces a cell decomposition of $\Core X$, the $k$-skeleton of which we denote by $(\Core X)_k$. Topologically, a critical Weinstein handle is diffeomorphic to $D^n \times D^n$, and the subset $D^n \times \{0\}$ is called the Lagrangian cocore of the Weinstein handle. For $0 \neq \varepsilon \in D^n$, we call $D^n \times \{\varepsilon\}$ a \emph{shifted Lagrangian cocore}.

\begin{defn}[Nearby Lagrangian cocore]
    An exact conical Lagrangian $L \subset X$ is a \emph{nearby Lagrangian cocore} if it is disjoint from all Lagrangian cocores of $X$, and is equal to a shifted Lagrangian cocore outside a compact subset of $X$.
\end{defn}

The purpose of this paper is to study a smooth version of Arnol{\cprime}d's problem on Lagrangian knots \cite{arnold1987first,eliashberg1993unknottedness,eliashberg1995topology,eliashberg1996local,polterovich2024lagrangian} in a general Weinstein sector.
\begin{quest}\label{q:unknotted}
    Is a nearby Lagrangian cocore in a Weinstein sector smoothly unknotted? I.e., is it smoothly isotopic (rel boundary) to a Lagrangian cocore, in the complement of all other Lagrangian cocores?
\end{quest}
Eliashberg--Polterevich \cite{eliashberg1996local} answered the symplectic version of \cref{q:unknotted} for $T^*\IR^2$, proving that nearby Lagrangian cocores can be unknotted through Hamiltonian isotopies. Côté--Dimitroglou Rizell \cite{cote2022symplectic} generalized their result to cotangent bundles of Riemann surfaces of arbitrary genus. In \cite{ekholm2018nearby}, Ekholm--Smith proved that nearby Lagrangian cocores in $X = T^\ast \IR^{n}$ are smoothly unknotted for $n \geq 4$.

We answer \cref{q:unknotted} in the affirmative in a number of new cases:
\begin{thm}\label{thm:unknotting_intro}
    Let $k \geq 1$ and let $n \geq 2k+2$. Let $X^{2n}$ be one of the following:
    \begin{itemize}
        \item $T^\ast \CP^2$, $T^\ast(S^5 \times S^5 \times \IR^k)$, $T^\ast(S^5 \times S^5 \times S^1)$, or
        \item $T^\ast(M^{n-k} \times N^k)$, where $N^k$ is a $k$-dimensional $K(\pi,1)$ manifold, and $M$ is an $(n-k)$-dimensional manifold such that $\pi_2(M) = 0$ and either $\pi_n(M)=0$ or its universal cover $\widetilde M$ has the homotopy type of a wedge sum of spheres $\bigvee_i S^{n-k_i}$ with $k_i \leq k$, or
        \item any plumbing of copies of $T^\ast((S^{n-k} \times \IR^k)^{\# i} \# (\varSigma^n)^{\# j})$ where $i,j \geq 0$, and $\varSigma$ is a homotopy $n$-sphere. 
    \end{itemize}
    Then any nearby Lagrangian cocore $L \subset X$ is smoothly unknotted.
\end{thm}
\begin{rem}
    \begin{enumerate}
        \item \cref{thm:unknotting_intro} is an amalgamation of \cref{cor:cp2,thm:smooth_iso,thm:conn_sum,thm:plumbing}. Examples of manifolds $M$ satisfying the second item in \cref{thm:unknotting_intro} include spheres, connected sums of copies of $S^{n-k-1} \times S^1$, real projective spaces, connected sums of copies of $\RP^{n-k-1} \times S^1$, higher dimensional lens spaces, and the octonionic projective plane $\OP^2$.
        \item When $X$ is as in \cref{thm:unknotting_intro}, it follows from \cite[Proposition 6.6]{eliashberg2020flexible} that \emph{flexible} nearby Lagrangian cocores $L \subset X$ are in fact Hamiltonian isotopic to a Lagrangian cocore in the complement of all Lagrangian cocores. We do not know whether any nearby Lagrangian cocore is flexible.
        \item A version of \cref{thm:unknotting_intro} where nearby Lagrangian cocores are smoothly unknotted not necessarily in the complement of the missed cotangent fiber also holds for the cotangent bundles of $S^2 \times S^2$, $S^2 \times \RP^2$, $\RP^2 \times \RP^2$, $S^6 \times S^6$, $S^6 \times \RP^6$, and $\RP^6 \times \RP^6$, see \cref{por:prod_spheres,rem:unkotted_not_in_complement}.
    \end{enumerate}
\end{rem}

\cref{thm:unknotting_intro} is an application of our main technical result, which we now describe: Let $\boldsymbol{C}$ denote the union of all Lagrangian cocores of $X$. Note that $\pi(L) \subset \Core X \smallsetminus \pi(\boldsymbol{C}) \simeq (\Core X)_{n-k}$, where $\pi$ is the retract from $X$ to its core, and where $n-k$ is the maximal index of a subcritical Weinstein handle in $X$. Furthermore, $\pi(\partial_\infty L) \subset \Core X \smallsetminus \pi(\boldsymbol{C})$ is a point which means that $\pi|_L$ restricts to a well-defined continuous map $L/\partial_\infty L \to \Core X \smallsetminus \pi(\boldsymbol{C})$ that we by abuse of notation keep denoting by $\pi|_L$.
\begin{asmpt}\label{asmpt:main_asmpt_intro}
    The Weinstein sector $X$ admits $MO\langle k+1\rangle$-Maslov data, and the nearby Lagrangian cocore $L \subset X$ admits relative $MO\langle k+1\rangle$-Maslov data.
\end{asmpt}

The main technical result is the following.

\begin{thm}[{\cref{thm:main}}]\label{thm:main_intro}
    Let $k \geq 1$ and let $n \geq 2k+2$. Suppose that $X^{2n}$ is a Weinstein sector with a chosen Weinstein handlebody decomposition such that the maximal index of a subcritical Weinstein handle is $\leq n-k$. Let $L \subset X$ be a nearby Lagrangian cocore such that $(X,L)$ satisfies \cref{asmpt:main_asmpt_intro}. The following composition is null-homotopic
    \[
    L/\partial_\infty L \overset{\pi|_L}{\longrightarrow} \Core X \smallsetminus \pi(\boldsymbol{C})  \longrightarrow (\Core X)_{n-k}/(\Core X)_{n-k-1}.
    \]
\end{thm}
\begin{rem}\label{rem:main_remark}
\begin{enumerate}
    \item \Cref{thm:main_intro} is a direct generalization of a result by Ekholm--Smith \cite[Theorem 1.1]{ekholm2018nearby} who proved it for $X = T^\ast \IR^{n}$.
    \item The assumption that a nearby Lagrangian cocore is disjoint from $\boldsymbol{C}$ is necessary, as is shown by performing Lagrangian surgery on the union of a cotangent fiber and the zero section inside $T^\ast Q$; this gives an exact Lagrangian $L$ such that $\pi|_L \colon L/\partial_\infty L \to Q$ has degree $1$.
    \item \Cref{asmpt:main_asmpt_intro} is a technical assumption and guarantees we are able to work with the spectral Donaldson--Fukaya category with coefficients in $MO\langle k+1\rangle$, see \cref{sec:idea_proof_intro} for details. \Cref{asmpt:main_asmpt_intro} holds whenever $X$ is stably polarizable, but it also holds for more general Weinstein sectors. For instance, there is a Weinstein structure on the total space of the vector bundle $(TS^6)^{\oplus 23} \to S^6$ that is not stably polarizable (cf.\@ \cite[Proposition 4.9]{alvarez2025arborealization}), but this Weinstein manifold and any nearby Lagrangian cocore in it satisfies \cref{asmpt:main_asmpt_intro} (see \cref{rem:non_pol_eg,lem:spheres_mg_brane_even_dim} for a general criterion).
\end{enumerate}
\end{rem}

Via the backwards Liouville flow, nearby Lagrangian cocores in $X$ correspond to exact Lagrangian fillings of the standard Legendrian unknot in a Darboux ball in the contact boundary of the subcritical part of $X$ (see \cref{lem:nearby_filling}). The following result gives topological restrictions on exact Lagrangian fillings of Legendrian slice knots (see \cref{defn:slice}) in the contact boundary of subcritical Weinstein sectors, and is a generalization of \cref{thm:main_intro}.
\begin{thm}[{\cref{thm:main_conc}}]\label{thm:main_conc_intro}
    Let $k\geq 1$ and $n\geq 2k+2$. Suppose that $X^{2n}_0$ is a subcritical Weinstein sector with a chosen Weinstein handlebody decomposition satisfying the hypotheses of \cref{thm:main_intro}. Suppose that $\varLambda \subset \partial_\infty X_0$ is a Legendrian knot contained in a Darboux ball that is Legendrian slice inside it. If $L \subset X_0$ is an exact Lagrangian filling of $\varLambda$, the following composition is null-homotopic
    \[
    L/\varLambda \overset{\pi|_L}{\longrightarrow} \Core X_0  \longrightarrow \Core X_0/(\Core X_0)_{n-k-1}.
    \]
\end{thm}
\begin{rem}
    \begin{enumerate}
        \item Because nearby Lagrangian cocores in $X$ correspond to exact Lagrangian fillings of $\varLambda_0 \subset \partial_\infty X_0$ (see \cref{lem:nearby_filling}), our main theorem \cref{thm:main_intro} is equivalent to \cref{thm:main_conc_intro} in the case $\varLambda$ is the standard (max-tb) Legendrian unknot. The simplest non-trivial Legendrian knot that \cref{thm:main_conc_intro} applies to is the max-tb Legendrian representative of the mirror of $9_{46}$, see e.g.\@ \cite{cornwell2016obstructions}.
        \item Combining \cref{thm:main_conc_intro} with \cref{thm:unknotting_intro} yields that there is a unique smooth isotopy class of exact Lagrangian fillings of Legendrian knots that contained in a Darboux ball and are Legendrian slice in it, in some collection of subcritical Weinstein sectors.
    \end{enumerate}
\end{rem}

\subsection{Idea of the proof of the main technical result}\label{sec:idea_proof_intro}

We explain the idea of the proof of \cref{thm:main_intro}. We denote by $X_0$ the subcritical part of $X$, and denote its contact boundary by $\partial_\infty X_0$. A nearby Lagrangian cocore $L \subset X$ is equal to the shifted Lagrangian cocore $C$ outside of a compact subset of $X$. Via the backwards Liouville flow, $L$ and $C$ become two exact Lagrangian fillings of the unknot in a Darboux ball $\varLambda_0 \subset \partial_\infty X_0$. Using classical Floer theory, this implies that $L$ must be a homotopy disk (and in fact diffeomorphic to a disk if $n\neq 4$), see \cref{lma:nearby_cocore_is_a_disk}. Attach a critical Weinstein handle to $X_0$ along $\varLambda_0$ and denote the resulting Weinstein sector by $\widehat X_0$. Attaching the core disk of the critical Weinstein handle to both $L$ and $C$ produces two exact Lagrangian homotopy spheres that we denote by $\widehat L$ and $\widehat C$, respectively.

From this point on we study objects supported on the exact Lagrangians $\widehat C$ and $\widehat L$ in the spectral wrapped Donaldson--Fukaya category $\sW(\widehat X_0;MO\langle k+1\rangle)$ with coefficients in $MO\langle k+1\rangle$; such spectral lifts of the wrapped Donaldson--Fukaya category were defined and studied in \cite{large2021spectral,asplund2024nearby,porcelli2023bordism,porcelli2024spectral,porcelli2025bordism,porcelli2025open-closed}. We use coefficients in the Thom spectrum $MO\langle k+1\rangle$ associated to the $k$-connected cover $O\langle k+1\rangle \to O$ of the orthogonal group. By the Pontryagin--Thom isomorphism, the homotopy groups of $MO\langle k+1\rangle$ is the bordism group of manifolds whose stable tangent bundles admit a reduction of their structure group to $O\langle k+1\rangle$; for instance $O\langle 1 \rangle = SO$ and $O\langle 2 \rangle = O\langle 3\rangle = Spin$. Namely, the homotopy groups of $MO\langle 1\rangle$ and $MO\langle 2\rangle$ are isomorphic to the oriented and spin bordism groups, respectively.

Now, \cref{asmpt:main_asmpt_intro} is a technical assumption (see \cref{dfn:maslov_data,dfn:relative_brane}) that guarantees that the exact Lagrangians $\widehat L$ support objects in $\sW(\widehat X_0;MO\langle k+1\rangle)$. Our main technical lemma is that they in fact support \emph{isomorphic} objects.

\begin{lem}[{\cref{lem:closures_quiv}}]
    Suppose that $\widehat X_0$ admits $MO\langle k+1\rangle$-Maslov data. Given any choice of $MO\langle k+1\rangle$-Maslov data on $\widehat L$, there exists a choice of $MO\langle k+1\rangle$-Maslov data on $\widehat C$ such that $\widehat L \cong \widehat C$ in $\sW(\widehat X_0;MO\langle k+1\rangle)$.
\end{lem}
\begin{rem}
	\cref{lem:closures_quiv} can also be proved by appealing to \cite[Theorem 1.2]{porcelli2025bordism} which uses obstruction theoretic techniques. Our proof relies instead on geometric arguments from \cite{asplund2024nearby} and \cref{subsec:subcrit_handles}, but is restricted by the geometry of the present situation.
\end{rem}

Note that $\widehat L$ is a homotopy spheres of sufficiently high dimension, so its stable tangent bundles admits a reduction of its structure groups to $O\langle k+1\rangle$, and a choice of such induces a choice of $MO\langle k+1\rangle$-fundamental class which we denote by $[\widehat L]$. Using a spectral analog of the open-closed map in Floer theory (see \cite{asplund2024nearby,porcelli2023bordism,porcelli2024spectral,porcelli2025open-closed}), we obtain the following.

\begin{lem}[{\cref{thm:oc_fund_classes}}]\label{lem:homologous_intro}
    There exists a choice of $MO\langle k+1\rangle$-fundamental class $[\widehat C]$, such that $[\widehat L] = [\widehat C] \in H_n(\widehat X_0;MO\langle k+1\rangle)$.
\end{lem}

Since the attaching locus $\varLambda_0$ of the critical Weinstein handle $\widehat X_0$ belongs to a Darboux ball in $\partial_\infty X_0$, the restriction of the retract $X_0 \to \Core X_0$ to $\varLambda_0$ is null-homotopic, and hence allows us to define a continuous extension to a ``fold map'' $\widehat \pi \colon \widehat X_0 \to \Core X_0$, see \cref{sec:fold_map}. 
\begin{proof}[Sketch of the proof of \cref{thm:main_intro}]
    Because $\widehat \pi(\widehat C) \subset \Core X_0$ is contractible, we have $\widehat \pi_\ast[\widehat C] = 0$, and hence \cref{lem:homologous_intro} implies that 
    \[
    \widehat \pi_\ast[\widehat L] = 0 \in \widetilde H_n(\Core X_0;MO\langle k+1\rangle).
    \]
    The homotopy type of the quotient $\Core X_0/(\Core X_0)_{n-k-1}$ is that of a wedge sum of copies of $S^{n-k}$, and by our assumption that $n\geq 2k+2$ it suffices to show that the composition
    \[
    f_{\widehat{L}} \colon \widehat L \longrightarrow \Core X_0 \longrightarrow \Core X_0/(\Core X_0)_{n-k-1} \simeq \bigvee_i S^{n-k} \overset{q_i}{\longrightarrow} S^{n-k},
    \]
    if null-homotopic for any $i$, see \cref{lem:HomotopyOfWedge}. Finally, by \cref{lma:fundamental_class}, there are isomorphisms
    \[
    \widetilde H_n(S^{n-k};MO\langle k+1\rangle) \cong \pi_k(MO\langle k+1\rangle) \cong \pi_k^{\mathrm{st}} \cong \pi_n(S^{n-k}), \quad n \geq 2k+2,
    \]
    such that the homotopy class of $f_{\widehat{L}}$ in $\pi_n(S^{n-k})$ coincides with the image of $\widehat \pi_\ast [\widehat L]$ under the composition of these isomorphisms, which finishes the proof.
\end{proof}
\begin{rem}\label{rem:wedge_sum_main_thm}
    In many cases relevant to \cref{thm:unknotting_intro}, we have $\Core X \smallsetminus \pi(\boldsymbol C) \simeq \bigvee_i S^{n-k_i}$. The proof of \cref{thm:main_intro} also reveals that every composition
    \[
    L/\partial_\infty L \overset{\pi|_L}{\longrightarrow} \Core X \smallsetminus \pi(\boldsymbol C) \simeq \bigvee_i S^{n_i} \overset{q_i}{\longrightarrow} S^{n-k_i}
    \]
    is null-homotopic.
\end{rem}

The idea to infer the smooth unknotting result \cref{thm:unknotting_intro} is to use \cref{thm:main_intro} and the consequences of its proof described in \cref{rem:wedge_sum_main_thm}, and to apply the $h$-principle \cite[Théorème d'existence]{haefliger1961plongements}.

\subsection{Outline}

In \cref{sec:maslov_data_disks} we give definitions of $R$-Maslov data on $X$ and $L$, and provide the necessary definitions of spaces of abstract disks required to orient relevant Floer theoretic moduli spaces of disks with $R$-orientations. We also identify $H\IZ$-Maslov data with the usual $\IZ$-brane structures in Floer theory. In \cref{sec:spectral_fuk} we review the definition of the wrapped Donaldson--Fukaya category with coefficients in a commutative ring spectrum. In \cref{sec:nearby_cocore} we prove that after capping off $X_0$ by a critical Weinstein handle, the nearby Lagrangian cocore and the shifted Lagrangian cocore support isomorphic objects in the spectral Donaldson--Fukaya category. In \cref{sec:proof} we prove our main results. In \cref{sec:fold_map} we give a definition of the fold map $\widehat \pi \colon \widehat X_0 \to \Core X_0$. In \cref{sec:detecting_null-homotopies} we explain how null-bordisms with tangential $O\langle k+1\rangle$-structures may be upgrade to null-homotopies in the present situation. In \cref{sec:geom_background} we provide some background on Weinstein sectors and handle attachments.

%% file: Rbranes.tex
\section{Maslov data and spaces of abstract disks}\label{sec:maslov_data_disks}
Throughout this section, we let $R$ be a commutative ring spectrum. In this section we review definitions of $R$-Maslov data for symplectic manifolds and Lagrangian manifolds therein. We also define abstract spaces of disks that are required to obtain $R$-orientations on appropriate index bundles of Cauchy--Riemann operators used in Floer theory.
\subsection{Grading and $R$-Maslov data}\label{sec:maslov}
   Denote by $\Pic(R)$ the maximal $\infty$-groupoid (equivalently the space) of invertible objects in the symmetric monoidal $\infty$-category of $R$-modules. It is an infinite loop space such that there is a fiber sequence
    \begin{equation}\label{eq:fiber_seq_pic}
        BGL_1(R) \longrightarrow \Pic(R) \longrightarrow \pi_0(\Pic(R)).
    \end{equation}
    In particular, $\varOmega\Pic(R) \cong GL_1(R)$. For any map of commutative ring spectra $S \to R$, there is a map of infinite loop spaces
    \[ \Pic(S) \longrightarrow \Pic(R).\]
    In particular, for any commutative ring spectrum $R$, there exists a map $\Pic(\IS) \to  \Pic(R)$ of infinite loop spaces. Recall that $J$-homomorphism is a map of infinite loop spaces
    \[ J \colon BO \longrightarrow \Pic(\IS).\]
    Pre-composing with the forgetful map $BU \to BO$ we obtain the \emph{complex} $J$-homomorphism
    \[ J_\IC \colon BU \longrightarrow \Pic(\IS).\]
    Consider the composition
    \begin{equation}\label{eq:maslov_fiber}
            \begin{tikzcd}[sep=5mm]
                B(U/O) \ar[r,"\simeq","\mathrm{Bott}_\IR"'] & B^3O \times B^2\IZ \rar{B^2J} & B^2\Pic(\IS) \rar & B^2\Pic(R)
            \end{tikzcd}
    \end{equation}
    where the $\mathrm{Bott}_\IR$ is the real Bott periodicity isomorphism defined as the twice delooping of the index map (cf.\@ \cite[Section 5]{atiyah1966ktheory}), and $B^2J$ is the double delooping of the $J$-homomorphism. Denote the homotopy fiber of \eqref{eq:maslov_fiber} by $B(U/O)^\# \to B(U/O)$. 
    \begin{defn}[$R$-Maslov data]\label{dfn:maslov_data}
        Let $X$ be a symplectic manifold. The choice of \emph{$R$-Maslov data} for $X$ is a lift of the composition
        \[
        \Gr(X) \colon X \overset{TX}{\longrightarrow} BU \longrightarrow B(U/O)
        \]
        to a map $\Gr(X)^\# \colon X \to B(U/O)^\#$. Equivalently, it is a choice of a null-homotopy of the following composition
        \begin{equation}\label{eq:maslov_data}
            \begin{tikzcd}[sep=5mm]
                X \rar{TX} & BU \rar & B(U/O) \rar{\simeq}[swap]{\mathrm{Bott}_\IR} & B^3O \times B^2\IZ \rar{B^2J} & B^2\Pic(\IS) \rar & B^2\Pic(R).
            \end{tikzcd}
    \end{equation}
    \end{defn}
    \begin{rem}\label{rem:polariz}
        A stable polarization of $X$ is a lift of $TX \colon X \to BU$ along the complexification map $BO \to BU$ to a map $X \to BO$. The canonical null-homotopy of the composition $BO \to BU \to B(U/O)$ then induces a choice of $R$-Maslov data for $X$.
    \end{rem}
    \begin{asmpt}\label{asmpt:bott_compat}
        The following diagram is commutative as a diagram of infinite loop spaces
        \begin{equation}\label{eq:bott_compat}
            \begin{tikzcd}[sep=scriptsize]
                \varOmega^2 BU \dar \ar[r,"\simeq","\ind_{\IC}"'] & BU \times \IZ \rar{J_{\IC}} \dar &\Pic(\IS) \\
                \varOmega^2B(U/O) \ar[r,"\simeq","\ind_\IR"'] & BO \times \IZ \urar[out=5,in=225][swap]{J}
            \end{tikzcd}.
        \end{equation}
    \end{asmpt}
    \begin{rem}
        It would be sufficient for our purposes if \eqref{eq:bott_compat} is a commutative diagram of $E_2$-spaces. It was proven that \eqref{eq:bott_compat} is commutative on the level of $E_1$-spaces in \cite[Proposition 2.18]{bonciocat2025floer}. A proof of \cref{asmpt:bott_compat} would follow if our Bott maps the same as those appearing in e.g.\@ \cite[Chapter I]{May} (cf.\@ maps appearing in \cite[(5.3.2)]{rognes2008galois} or \cite[p.\@ 161]{rognes2003smooth}).
        
        Under \cref{asmpt:bott_compat}, the map \eqref{eq:maslov_data} is homotopic to the composition
        \[
            \begin{tikzcd}[sep=7mm]
                X \rar{TX} & BU \rar{\simeq}[swap]{\mathrm{Bott}_\IC} & B^3U \times B^2\IZ \ar[r,"B^2J_{\IC}"] & B^2\Pic(\IS) \rar & B^2\Pic(R)
            \end{tikzcd}.
        \]
    \end{rem}
    \begin{rem}\label{rem:twice_chern}
        In case $R$ is any commutative ring spectrum such that $\pi_0(\Pic(R)) = \IZ$, the composition 
        \begin{equation}\label{eq:twice_chern_map}
            \begin{tikzcd}[sep=scriptsize]
                X \rar{TX} & BU \rar & B(U/O) \rar{\text{\eqref{eq:maslov_fiber}}} & B^2\Pic(R) \rar & B^2\IZ
            \end{tikzcd}
        \end{equation}
        is precisely represented by $2c_1(X) \in H^2(X;\IZ)$ (cf.\@ \cite[Remark 5.2]{ganatra2024microlocal}, \cite[Theorem 4.18]{cote2022perverse}). Choosing a null-homotopy for this map is equivalent to choosing a quadratic complex volume form on $TX$, i.e.\@, a trivialization of $(\varLambda_\IC^{\mathrm{top}}TX)^{\otimes 2}$, see \cite[(11j)]{seidel2008fukaya}. In particular, the collection of homotopy classes of quadratic complex volume forms on $TX$ forms a torsor over $H^1(X;\IZ)$.
    \end{rem}

    \begin{lem}\label{lem:complex_ori_maslov}
            Let $R$ be a complex oriented commutative ring spectrum. If $2c_1(X) = 0$, then $X$ admits a canonical choice of $R$-Maslov data.
    \end{lem}
    \begin{proof}
        Since $R$ is complex oriented, the unit map $\IS \to R$ factors through $MU$. The composition
        \[
        BU \overunderset{\simeq}{\mathrm{Bott}_\IC}{\longrightarrow} B^3U \times B^2\IZ \longrightarrow B^3O \times B^2\IZ \longrightarrow B^2\IZ
        \]
        is classified by $2c_1$, and consequently we have an induced commutative diagram
        \[
        \begin{tikzcd}[sep=scriptsize]
            && B^3 U \dar \\
            & B^3U \times B\IZ/2 \urar[bend left=20,dashed]{g} \rar \dar & B^3O \dar \\
            \hofib(2c_1) \rar \urar[bend left=20,dashed]{f} & B^3U \times B^2\IZ \drar[bend right=20] \rar & B^3O \times B^2\IZ \dar \\
            && B^2\IZ
        \end{tikzcd}.
        \]
        The map $B^3U \times B^2\IZ \to B^2\IZ$ in the above diagram is given by the projection composed with the twice delooping of $\IZ \xrightarrow{\cdot 2} \IZ$, and the composition $BU \xrightarrow{\simeq} B^3U \times B^2 \IZ \to B^2\IZ$ in this diagram is the classifying map for $2c_1$, which yields the first lift $f$. The second lift $g$ exists, because the map $B^3U \times B^2 \IZ \to B^3O \times B^2\IZ$ is the product of the third delooping of the forgetful map $U \to O$, and the twice delooping of the map $\IZ \xrightarrow{\cdot 2} \IZ$. Consequently, the map $B^3U \times B\IZ/2 \to B^3O$ factors as $B^3U \times B\IZ/2 \to B^3U \to B^3O$. %\sayYmargin{I don't understand the last sentence.}\sayJmargin{what about now?}

        Thus, we end up with the commutative diagram
        \[
        \begin{tikzcd}[sep=scriptsize]
            & B^3 U \dar \\
            & B^3 O \rar \dar & B^3GL_1(MU) \dar \\
            \hofib(2c_1) \ar[uur,bend left=20,dashed] \rar & B^3O \times B^2\IZ \rar & B^2\Pic(MU)
        \end{tikzcd},
        \]
        where the composition $B^3U \to B^3O \to B^3GL_1(MU)$ is canonically null-homotopic by the complex orientation on $MU$. Consequently the composition $\hofib(2c_1) \to B^2\Pic(MU)$ is canonically null-homotopic, and the result follows.

    \end{proof}
    
\subsection{Grading and $R$-Maslov data on a Lagrangian}\label{sec:maslov_data_lagrangian} 
    Let $L \subset X$ be a Lagrangian submanifold inside a symplectic manifold $X$.
    Suppose that $X$ equipped with $R$-Maslov data $\Gr(X)^\# \colon X \to B(U/O)^\#$, see \cref{dfn:maslov_data}. The composition
    \begin{equation}\label{eq:maslov_resL}
        L \overset{TL}{\longrightarrow} BO \longrightarrow BU \longrightarrow B(U/O) \overset{\text{\eqref{eq:maslov_fiber}}}{\longrightarrow} B^2\Pic(R)
   \end{equation} 
   is canonically null-homotopic since the composition $BO \to BU \to B(U/O)$ is. The restriction of the $R$-Maslov data on $X$ to $L$ and the canonically null-homotopic map \eqref{eq:maslov_resL} therefore correspond to two lifts
    \begin{equation}\label{eq:two_null_homotopies}
        \begin{tikzcd}[sep=1.2cm]
            L \ar[r, shift left, "TL^\#"] \ar[r, shift right, "\Gr(X)^\#|_L"'] & B(U/O)^\# 
        \end{tikzcd}.
    \end{equation}
    They determine a lift
    \begin{equation}\label{eq:lift_maslov_data_L}
        (TL^\# - \Gr(X)^\#|_L) \colon L \longrightarrow B\Pic(R).
    \end{equation}
     \begin{defn}[$R$-Maslov data for a Lagrangian]\label{dfn:R_brane_struct}
        A choice of \emph{$R$-Maslov data} for $L$ is a null-homotopy of $TL^\# - \Gr(X)^\#|_L$. Equivalently, it is a homotopy between the two maps $TL^\#$ and $\Gr(X)^\#|_L$.
    \end{defn}
    \begin{defn}[$R$-brane]\label{eq:r-brane}
        An $R$-brane is an exact conical Lagrangian submanifold $L \subset X$ that stays away from the symplectic boundary of $X$, that is furthermore equipped with a choice of $R$-Maslov data.
    \end{defn}
    \begin{rem}
        The notion of $R$-Maslov data in \cref{dfn:R_brane_struct} was called ``secondary Maslov data'' in \cite{nadler2020sheaf}, and ``brane structure'' in \cite{jin2020microlocal} in the case when $X$ is stably polarizable.
    \end{rem}

    \begin{rem}\label{rem:maslov}
       Let $R$ be any commutative ring spectrum such that $\pi_0(\Pic(R)) = \IZ$. The composition \eqref{eq:twice_chern_map} represents $2c_1(X) \in H^2(X;\IZ)$, and an $R$-Maslov data on $X$ gives a choice of a primitive $m\in C^1(X;\IZ)$ of $2c_1(X)$. When pulling back the map $X \to B^2\IZ$ to $L$, we obtain two null-homotopies of $L \to B^2\IZ$ corresponding to primitives $m|_L, m_0 \in C^1(L;\IZ)$ such that $2c_1(X) = d(m|_L) = d(m_0)$, whose difference $m_0-m|_L \in H^1(L;\IZ)$ is classified by the composition
       \[
       L \xrightarrow{TL^\# - \Gr(X)^\#|_L} B\Pic(R) \longrightarrow B\IZ.
       \]
       This difference is the integer valued Maslov class $\mu_L \in H^1(L;\IZ)$, which gives an integer-valued Maslov index in the sense of \cite{robbin1993maslov}. A choice of $R$-Maslov data on $L$ thus in particular implies that the Maslov class of $L$ vanishes.
    \end{rem}
    
    \begin{rem} Existence of Maslov data as an obstruction to defining $R$-linear Fukaya categories was discussed briefly by Lurie \cite[Section 1.3.1]{lurie2015rotation}. $R$-linear microlocal sheaf categories were studied by Jin--Treumann \cite{JinTr} and Nadler--Shende \cite{nadler2020sheaf}. 
    \end{rem}
    \begin{rem}
        Porcelli--Smith \cite{porcelli2023bordism,porcelli2024spectral} defined a spectral Donaldson--Fukaya category associated to a pair of based spaces $(\varTheta,\varPhi)$ fitting into a commutative diagram
        \[
        \begin{tikzcd}[sep=scriptsize]
            \varTheta \rar \dar & BO \dar{-\otimes \IC} \\
            \varPhi \rar& BU
        \end{tikzcd}.
        \]
        Our notion of $R$-Maslov data on $X$ and $R$-Maslov data on $L$ corresponds to choosing
        \[
        \varPhi = \hofib(BU \to B(U/O) \xrightarrow{\text{\eqref{eq:maslov_fiber}}} B^2\Pic(R)) \quad \text{and} \quad \varTheta = BO,
        \]
        and requiring that the following diagram is (homotopy) commutative
        \[
        \begin{tikzcd}[sep=scriptsize]
            L \dar[hook] \rar & \varTheta \ar[r,"="] \dar & BO \dar{-\otimes \IC} \\
            X \rar &\varPhi \rar & BU
        \end{tikzcd},
        \]
        such that the composition of the top row and the bottom row are homotopic to $TL$ and $TX$, respectively.
        This is precisely the definition of $X$ admitting a $\varPhi$-structure, and $L$ admitting a $\varTheta$-brane structure in the sense of \cite[Section 1.2]{porcelli2024spectral}.
    \end{rem}
    \begin{rem}
        The relation between existence of $\IS$-Maslov data on a Weinstein manifold and the existence of an arborealization of the skeleton was recently studied by Álvarez-Gavela--Large--Ward \cite{alvarez2025arborealization}, see also \cite{alvarez2020positive}.
    \end{rem}

    \begin{notn}
	    \begin{enumerate}
	        \item Fix a map of infinite loop spaces $G \to O$. Denote by $G^c$ the homotopy fiber product $U \times^h_O G$.
            \item Denote by $MG$ and $MG^c$ the Thom spectra $\text{Thom}(BG \to BO)$ and $\text{Thom}(BG^c \to BU \to BO)$, respectively.
	    \end{enumerate}
    \end{notn}
    \begin{defn}
	    We say that $X$ admits a \emph{tangential $MG$-Maslov data} if there exists a lift
	    \[
            \begin{tikzcd}
                && X \ar[lld,dashed,bend right=15,swap,"\exists"] \dar\\
                B^3G \rar & B^3 O  \rar& B^3O \times B^2 \IZ 
            \end{tikzcd}
	    \]
    \end{defn}
    \begin{lem}\label{lem:suff_tangential_maslov}
        Let $k\geq 1$ and let $G = O\langle k\rangle$ denote the $(k-1)$-connected cover of $O$. A Liouville sector $X$ admits tangential $MO\langle k\rangle$-Maslov data if for every $1 \leq j \leq k+3$
        \begin{align*}
            H^j(X;\IZ) &= 0 \text{ if } j \equiv 0,2 \pmod 8, \text{ and}\\
            H^j(X;\IZ/2) &= 0 \text{ if } j \equiv 3,4 \pmod 8.
        \end{align*}
    \end{lem}
    \begin{proof}
        First, since $H^2(X;\IZ) = 0$, we in particular have $2c_1(X) = 0$, which implies first that $X \to B^3O \times B^2\IZ$ admits a lift to $X \to B^3O$. To further lift this map to $X \to B^3O\langle k\rangle$, the obstruction belongs to $H^j(X;\pi_{j-3}O)$ for $1 \leq j \leq k+3$.
    \end{proof}
    \begin{rem}\label{rem:non_pol_eg}
        There is a Weinstein structure on the total space $X$ of the vector bundle $(TS^6)^{\oplus 23} \to S^6$ such that $X$ is neither arborealizable, nor stably polarizable by \cite[Proposition 4.9]{alvarez2025arborealization}. However, $X$ certainly admits tangential $MO\langle k\rangle$-Maslov data for any $k\geq 1$ by \cref{lem:suff_tangential_maslov}.
    \end{rem}
    \begin{lem}
       Any choice of tangential $MG$-Maslov data on $X$ canonically induces a choice of a $MG$-Maslov data on $X$.
    \end{lem}
    \begin{proof}
       This follows because of the $MG$-orientation on $MG$ yields that the composition
       \[
       B^3 G \longrightarrow B^3 O \longrightarrow B^3GL_1(MG) \longrightarrow B^2 \Pic(MG)
       \]
       is canonically null-homotopic.
    \end{proof}

    Note that a choice of tangential $MG$-Maslov data on $X$ induces a lift of $X \to BU \xrightarrow{\simeq} B^3U \times B^2\IZ$ to a map
    \begin{equation}\label{eq:complex_tangential}
        X \longrightarrow B^3G^c \times B \IZ/2,
    \end{equation}
    since we have a homotopy pullback square
    \[
        \begin{tikzcd}[sep=scriptsize]
            B^3 G^c \times B\IZ/2 \rar \dar \drar[phantom,description,very near end, "\lrcorner"] & B^3 G \dar \\
            B^3U \times B^2\IZ \rar & B^3O \times B^2\IZ
        \end{tikzcd}.
    \]

    \begin{lem}\label{lem:bpic_factors}
    Suppose that $X$ is a Liouville sector equipped with a choice of tangential $MG$-Maslov data and let $L \subset X$ be a Lagrangian. Then the corresponding map $TL^\# - \Gr(X)^\#|_L$ as in \eqref{eq:lift_maslov_data_L} factors as
    \[
    L \longrightarrow B^2 (O/G^c) \times B\IZ \times B\IZ/2 \longrightarrow B\Pic(MG).
    \]
    \end{lem}
    \begin{proof}
        First, since $BO \to BU \to B(U/O)$ is a fiber sequence, we obtain the following commutative diagram
        \[ 
        \begin{tikzcd}[row sep=scriptsize]
            L \rar{TL} \ar[d,hook] & BO \rar \dar{c} & B^2(O/U) \times B\IZ/2 \dar\\
            X \rar{TX}  &BU \rar{\simeq} & B^3 U \times B^2\IZ
        \end{tikzcd}.
        \]
        Via the choice of tangential $MG$-Maslov data on $X$, the map \eqref{eq:complex_tangential} leads to a commutative diagram
        \[
        \begin{tikzcd}[sep=scriptsize]
            L \rar \dar & B^2(O/U) \times B\IZ/2 \dar \\
            B^3G^c \times B\IZ/2 \rar & B^3U \times B^2\IZ
        \end{tikzcd},
        \]
        and hence an induced map to the homotopy fiber product:
        \[
        L \longrightarrow B^2(O/G^c) \times B\IZ \times \IZ/2 \simeq (B^2(O/U) \times B\IZ/2) \times^h_{B^3U \times B^2\IZ} (B^3G^c \times B\IZ/2).
        \]
        Consider the homotopy commutative diagram
        \begin{equation}\label{eq:induced_bpic_dia}
            \begin{tikzcd}[sep=scriptsize]
                 B^2(O/G^c) \times B\IZ \times B\IZ/2  \rar[dashed] \dar & B\Pic(MG) \dar\\
                 B^3 G^c \times B\IZ/2 \rar \dar & \ast \dar\\
                 B^3O \times B^2\IZ \rar & B^2\Pic(MG)
            \end{tikzcd},
        \end{equation}
        where the vertical compositions are fiber sequences. The middle horizontal arrow is obtained using the fact that $B^3G^c \times B\IZ/2 \to B^3O \times B^2\IZ \to B^2\Pic(MG)$ is homotopic to the composition
        \[
        B^3G^c \times B\IZ/2 \longrightarrow B^3 G \longrightarrow B^3O \longrightarrow B^3GL_1(MG) \longrightarrow B^2\Pic(MG),
        \]
        which is null-homotopic by the $MG$-orientation on $MG$. The induced maps on fibers, given by the top horizontal arrows in \eqref{eq:induced_bpic_dia}, coincides with difference between the two null-homotopies of the composition 
        \[ B^2(O/G^c) \times B\IZ \times B\IZ/2 \longrightarrow B^2 \Pic(MG).\]
        It follows that composition 
        \[ L \longrightarrow B^2(O/G^c) \times B\IZ \times B\IZ/2 \longrightarrow B\Pic(MG).\]
        coincides with the difference $TL^\# - \Gr(X)^\#|_L$.
    \end{proof}
 
    \begin{lem}\label{lem:spheres_mg_brane_even_dim}
        Let $k \geq 1$ and suppose $n \geq k+1$, and that $n$ is even. Suppose $X$ is a Liouville sector, and let $L \subset X$ a homotopy $n$-sphere. Suppose that the infinite loop map $G \to O$ is an isomorphism on $\pi_i$ for $i \geq k-1$, and assume that $X$ admits a choice of tangential $MG$-Maslov data. Then $L$ admits $MG$-Maslov data.
    \end{lem}
    \begin{proof}
        Similar to the diagram \eqref{eq:induced_bpic_dia}, there is a commutative diagram
        \[
        \begin{tikzcd}[sep=scriptsize]
            B^2(O/G^c) \dar \rar[dashed] & B^2GL_1(MG) \dar & \\
            B^3G^c \rar \dar & \ast \dar \\
            B^3O \rar{B^3J_{\IR}} & B^3GL_1(MG)
        \end{tikzcd}
        \]
        that admits a functor to the diagram \eqref{eq:induced_bpic_dia} via the obvious maps. The assumption that $L$ is a homotopy $n$-sphere yields a lift of $L \to B^2(O/G^c) \times B\IZ \times B\IZ/2$ to give a commutative diagram
        \begin{equation}\label{eq:diagram_lift_bz}
            \begin{tikzcd}[sep=scriptsize]
                & B^2O \rar{B^2J_{\IR}} \dar & B^2GL_1(MG) \dar[sloped,phantom]{\simeq} \\
                & B^2(O/G^c) \rar \dar & B^2GL_1(MG) \dar \\
                L \urar[dashed,bend left=20] \rar & B^2(O/G^c) \times B\IZ \times B\IZ/2 \rar & B\Pic(MG)
            \end{tikzcd}
        \end{equation}
        Next, since $n$ is even, and because $G \to O$ is an isomorphism on $\pi_i$ for $i\geq k-1$, it follows that $\pi_n(B^3G^c) \cong 0$, and consequently that the map $L \to B^2(O/G^c)$ admits a further lift to $L \to B^2O$. This map admits an even further lift $L \to B^2G$, since $L \to B^2O \to B^2(O/G)$ is null-homotopic since $G \to O$ is an isomorphism on $\pi_i$ for $i\geq k-1$. To summarize we have the following commutative diagram:
        \[
        \begin{tikzcd}[sep=scriptsize]
                & B^2G \dar & \\
                & B^2O \rar{B^2J_{\IR}} \dar & B^2GL_1(MG) \dar[sloped,phantom]{\simeq} \\
                & B^2(O/G^c) \rar \dar & B^2GL_1(MG) \dar \\
                L \ar[ruuu,dashed,bend left=20] \rar & B^2(O/G^c) \times B\IZ \times B\IZ/2 \rar & B\Pic(MG)
            \end{tikzcd}
        \]
        By the $MG$-orientation on $MG$, the induced map $L \to B^2 G \xrightarrow{B^2J_{\IR}} B^2GL_1(MG)$ is null-homotopic, showing that $L$ admits $MG$-Maslov data by \eqref{eq:diagram_lift_bz}.
    \end{proof}
    \begin{lem}\label{lem:polariz_mo_maslov}
        Let $k\geq 1$ and suppose $n\geq k+1$. Suppose that the infinite loop map $G \to O$ is an isomorphism on $\pi_i$ for $i \geq k-1$. If $X$ is a stably polarized Liouville sector, and $L \subset X$ is a homotopy $n$-sphere, then $L$ admits $MG$-Maslov data.
    \end{lem}
    \begin{proof}
        Under the assumption that $X$ is stably polarized, the induced map $TL^\# - \Gr(X)^\#|_L$ in \eqref{eq:lift_maslov_data_L} factors as
        \[
        \begin{tikzcd}[row sep=scriptsize,column sep=1cm]
            L \rar{\sG_L} & U/O \rar{\varOmega \mathrm{Bott}_\IR} & B^2O \times B\IZ \rar{BJ_\IR} & B\Pic(\IS) \rar & B\Pic(MG),
        \end{tikzcd}
        \]
        where $\sG_L$ is the stable Lagrangian Gauss map, see \cite[Section 5.12]{asplund2024nearby}. Since $L$ is a homotopy $n$-sphere, this map is homotopic to a composition $L \to B^2O \to B^2GL_1(\IS) \to B^2GL_1(MG)$. The composition $L \to B^2O \to B^2(O/G)$ is null-homotopic since $\pi_{n-2} O/G = 0$, showing that there is an induced lift $L \to B^2G$:
        \[
        \begin{tikzcd}[sep=scriptsize]
            & B^2G \dar&& \\
            L \rar \urar[bend left=20,dashed] & B^2O \rar \dar & B^2GL_1(\IS) \rar & B^2GL_1(MG) \\
            & B^2(O/G)&&
        \end{tikzcd}.
        \]
        Since $MG$ is $MG$-oriented, the composition $MG\to B^2O \to B^2GL_1(\IS) \to B^2GL_1(MG)$ is null-homotopic, showing that $L$ admits $MG$-Maslov data.
    \end{proof}

    \subsubsection{Maslov data over the integers}\label{sec:ori_data_integers}
        The purpose of this section is to discuss $H\IZ$-Maslov data and compare it to the classical discussion on relative pin structures and Maslov classes used to define (integer graded) Floer theory with integral coefficients (cf.\@ \cite[Section 11]{seidel2008fukaya}); in the presence of a stable polarization a similar discussion is found in \cite[Section 5.4]{asplund2024nearby}.

        Let $X$ be a Liouville sector and let $L \subset X$ denote an exact conical Lagrangian that stays away from the symplectic boundary of $X$. Fix a background class $b \in H^2(X;\IZ/2)$, and assume that $X$ admits a quadratic complex volume form $\eta^2_X$ that is a trivialization of $(\varLambda_\IC^{\mathrm{top}} TX)^{\otimes 2}$. Recall that the space quadratic complex volume forms on $X$ forms a torsor over $H^1(X;\IZ)$ with any other such form given by $m \cdot \eta^2_X$ for some $m \colon X \to K(\IZ,1) = S^1 \subset \IC$ classifying an element of $H^1(X;\IZ)$, that we by abuse of notation also denote by $m$. 
        \begin{asmpt}\label{asmpt:sw_lag}
            The second Stiefel--Whitney class of $L$ agrees with the pullback of $b$ under the inclusion $i \colon L \hookrightarrow X$.
        \end{asmpt}
        Fix a triangulation on $X$ such that $L$ is a subcomplex. Denote the $3$-skeletons of $X$ and $L$ determined by this triangulation by $X[3]$ and $L[3]$, respectively. Let $E_b$ be an orientable vector bundle on $X[3]$, such that $w_2(E_b) = b|_{X[3]}$; there is a unique such vector bundle up stable isomorphism. We fix a null-homotopy of the classifying map $L[3] \to K(\IZ/2,2)$ of $w_2(E_b|_{L[3]}) - b|_{L[3]}$. Note that the obstruction to the existence of a pin structure is the second Stiefel--Whitney class. It vanishes for the vector bundle $E_b|_{L[3]} \oplus TL|_{L[3]}$ because of the orientability of $E_b$ and by \cref{asmpt:sw_lag}. The following definition is standard, see e.g.\@ \cite{solomon2006intersection,abouzaid2012wrapped} and cf.\@ \cite[Chapter 8]{fukaya2009lagrangian}.
        \begin{defn}[Relative pin structure]\label{dfn:rel_pin}
            A \emph{relative pin structure} on $L \subset X$ with respect to the background class $b\in H^2(X;\IZ/2)$ is a choice of a pin structure on the vector bundle $E_b|_{L[3]} \oplus TL|_{L[3]}$.
        \end{defn}
        The quadratic complex volume form $\eta^2_X \in \varGamma((\wedge^{\mathrm{top}}_\IC T^*X)^{\otimes 2})$ defines a squared phase map on the Lagrangian $L$
        \begin{align}\label{eq:lag_phase}
            \alpha_L \colon L &\longrightarrow S^1\\
            x &\longmapsto \mathrm{arg}(\eta^2_X (v_1 \wedge \dots \wedge v_n)),\nonumber
        \end{align}
        where $(v_1,\dots,v_n)$ is any basis for $T_xL$.
        \begin{defn}\label{dfn:lag_gr}
            A \emph{grading data} on $L \subset X$ with respect to the quadratic complex volume form $\eta^2_X \in  \varGamma((\wedge^{\mathrm{top}}_\IC T^*X)^{\otimes 2})$ is a lift of the squared phase map \eqref{eq:lag_phase} to a map 
            \[\alpha^\# \colon L \longrightarrow \IR\] along the exponential $t \mapsto \exp(2\pi i t)$.
        \end{defn}
        \begin{defn}[$\IZ$-brane structure]
            Let $L \subset X$ be a Lagrangian satisfying \cref{asmpt:sw_lag}. A choice of \emph{$\IZ$-brane structure} on $L$ with respect to the background classes $b \in H^2(X;\IZ/2)$ and quadratic complex volume form $\eta^2_X \in \varGamma ((\wedge^{\mathrm{top}}_\IC T^*X)^{\otimes 2})$ on $X$ is a choice of a relative pin structure with respect to $b\in H^2(X;\IZ/2)$ and a choice of grading data $\alpha^\#$ on $L$ with respect to $\eta^2_X$.
        \end{defn}
        
        \begin{lem}\label{lem:hz_maslov}
            Let $X$ be a Liouville sector that is equipped with a quadratic complex volume form $\eta^2_X \in \varGamma ((\wedge^{\mathrm{top}}_\IC T^*X)^{\otimes 2})$ and let $L \subset X$ be a Lagrangian that stays away from the symplectic boundary of $X$.
            \begin{enumerate}
                \item Homotopy classes of $H\IZ$-Maslov data on $X$ are in one-to-one correspondence with elements of $H^2(X;\IZ/2) \oplus H^1(X;\IZ)$.
                \item Assume that we have chosen $H\IZ$-Maslov data on $X$ that corresponds to a pair $(b,m) \in H^2(X;\IZ/2) \oplus H^1(X;\IZ)$ under the identification in item (i) such that $L$ satisfies \cref{asmpt:sw_lag}. Then a choice of $H\IZ$-Maslov data for $L$ in the sense of \cref{dfn:R_brane_struct}, is equivalent to a choice of $\IZ$-brane structure with respect to the pair $(b, m\cdot \eta^2_X)\in H^2(X;\IZ/2) \oplus \varGamma((\wedge^{\mathrm{top}}_\IC T^*X)^{\otimes 2})$.
            \end{enumerate}
        \end{lem}
        \begin{proof}
            \begin{enumerate}
                \item Since $H\IZ$ is complex oriented, \cref{lem:complex_ori_maslov} shows that $X$ admits $H\IZ$-Maslov data. In particular, the following diagram is commutative
                \[
                \begin{tikzcd}[sep=scriptsize]
                    \hofib(2c_1) \dar \rar \ar[rr,bend left=20] & B^3U \times B^2\IZ \dar & \ast \dar \\
                    B(U/O) \rar & B^3O \times B^2\IZ \rar & B^2\Pic(H\IZ)
                \end{tikzcd}.
                \]
                The quadratic complex volume form $\eta^2_X$ provides a null-homotopy of the projection
                \begin{equation}\label{eq:null_htpy_2c1}
                    X \longrightarrow BU \longrightarrow B(U/O) \overset{\eqref{eq:maslov_fiber}}{\longrightarrow} B^2\Pic(H\IZ) \longrightarrow K(\IZ,2),
                \end{equation}
                which classifies $2c_1(X) \in H^2(X;\IZ)$, showing that $X \to BU$ factors as $X \to \hofib(2c_1) \to BU$. Now, any other choice of $H\IZ$-Maslov data, i.e., null-homotopy of the composition
                \[
                X \longrightarrow \hofib(2c_1) \longrightarrow B(U/O) \longrightarrow B^2\Pic(H\IZ)
                \]
                induces a map $X \to F \to B\Pic(H\IZ)$. The map $B\Pic(H\IZ) \to K(\IZ,1)$ admits a section since the relevant obstruction belongs to $H^2(S^1;\IZ/2) \cong 0$. (We caution the reader that the induced splitting $B\Pic(H\IZ) \cong K(\IZ,1) \times K(\IZ/2,2)$ as spaces, is \emph{not} a splitting as loop spaces.) Therefore any map $X \to K(\IZ,1)$ admits a lift to a map $X \to B\Pic(H\IZ)$, and the set of such lifts is in one-to-one correspondence with elements in $H^2(X;\IZ/2)$. We conclude $[X,B\Pic(H\IZ)] \cong H^2(X;\IZ/2) \oplus H^1(X;\IZ)$, thus finishing the proof.
                \item Via item (i) above we consider a choice of $H\IZ$-Maslov data on $X$ corresponding to the pair of classes $(b,m) \in H^2(X;\IZ/2) \oplus H^1(X;\IZ)$.
                
                Letting $F\coloneqq \hofib(\hofib(2c_1)\to B(U/O))$, we obtain the commutative diagram
                \[
                \begin{tikzcd}[sep=scriptsize]
                    F \rar[dashed] \dar & B\Pic(H\IZ) \dar \\
                    \hofib(2c_1) \rar \dar & \ast \dar \\
                    B(U/O) \rar & B^2\Pic(H\IZ)
                \end{tikzcd}.
                \]
                Noticing that the map $BO \to BU$ factors through $\hofib(2c_1)$, we obtain an induced map $BO \to F$. Furthermore by definition the homotopy fiber of $F \to \hofib(2c_1)$ is homotopy equivalent to $U/O$ and since the homotopy cofiber of $F \to \hofib(2c_1)$ is equal to $B(U/O)$, the following diagram is commutative
                \begin{equation}\label{eq:lift_umodo}
                    \begin{tikzcd}[sep=scriptsize]
                        U/O \rar[equals] \dar &U/O \rar[dashed] \dar & B\Pic(H\IZ) \dar[sloped,phantom]{\simeq} \\
                        BO \rar &F \rar[dashed] \dar & B\Pic(H\IZ) \dar \\
                        &\hofib(2c_1) \rar \dar & \ast \dar \\
                        &B(U/O) \rar & B^2\Pic(H\IZ)
                    \end{tikzcd}.
                \end{equation}
                The induced map $F \to B\Pic(H\IZ)$ measures by definition the difference between the null-homotopy induced by the null-homotopy of $\hofib(2c_1) \to B^2\Pic(H\IZ)$, and the canonical null-homotopy of $F \to \hofib(2c_1) \to B(U/O)$.
                
                Next, denote the compositions $BO \to B\Pic(H\IZ) \to B\IZ$ and $U/O \to B^2O \times B\IZ \to B\IZ$ by $\mu_{BO}$, and $\mu$, respectively. Note that the map $\mu$ classifies the universal Maslov class. The upper commutative rectangle in \eqref{eq:lift_umodo} induces a commutative diagram
                \[
                \begin{tikzcd}[sep=scriptsize]
                    \hofib(\mu) \rar \dar & B^2O \rar{B^2J} & B^2GL_1(H\IZ) \dar \\
                    \hofib(\mu_{BO}) \ar[rr] && B^2GL_1(H\IZ)
                \end{tikzcd}.
                \]
                It follows from \cite[Lemma 11.7]{seidel2008fukaya} that the bottom horizontal arrow is the pullback of the universal second Stiefel--Whitney class $w_2 \colon BO \to B^2\IZ$ under the map $\hofib(\mu_{BO}) \to BO$. Recall by definition that our choice of $H\IZ$-Maslov data on $X$ corresponds to the background class $b\in H^2(X;\IZ/2)$. This choice also induces as explained above a composition $L \to BO \to B\Pic(H\IZ)$. Finally, a choice of $H\IZ$-Maslov data for $L$, i.e., a null-homotopy of this composition, implies that $L \to BO \to B\Pic(H\IZ) \to B \IZ$ is null-homotopic. This null-homotopy corresponds to grading data on $L$ in the sense of \cref{dfn:lag_gr}. We therefore obtain that $L \to BO$ factors as $L \to \hofib(\mu_{BO}) \to BO$, and consequently that the composition
                \[
                L \longrightarrow \hofib(\mu_{BO}) \longrightarrow BO \overset{w_2}{\longrightarrow} B^2GL_1(H\IZ)
                \]
                is null-homotopic. By construction, it follows that this composition is the classifying map for $w_2(TL) - b|_L \in H^2(L;\IZ/2)$, a null-homotopy of which corresponds to a relative pin structure on $L$ with respect to $b \in H^2(X;\IZ/2)$.%\sayJmargin{still a little sketchy, but I dont think we need to identify $L \to BO$ as $TL-E_b|_L$}

            \end{enumerate}
        \end{proof}

\subsection{Space of abstract disks}\label{sec:abstract_disks}
	\begin{notn}
        \begin{itemize}
            \item Let $R$ be a commutative ring spectrum.
            \item Let $(X,\lambda)$ be a Liouville sector.
            \item Recall that  $c \colon BO \to BU$ denotes the complexification map, $\Gr \colon BU \to B(U/O)$ denotes the induced quotient map, and $\Gr(X)$ denotes the composition $X \xrightarrow{TX} BU \xrightarrow{\Gr} B(U/O)$.
        \end{itemize}
	\end{notn}
	\begin{defn}
    Define 
        \[	
        \sD \coloneqq
        \hofib \left(BU \longrightarrow \sL BU \longrightarrow \sL B(U/O) \right).
        \]
    \end{defn}
	Note that there is a commutative diagram
    \[
    \begin{tikzcd}[sep=scriptsize]
        \sD \rar \dar & \sL BO \rar \dar{\sL c} & \ast \dar \\
        BU  \rar & \sL BU \rar{\sL \Gr} & \sL B(U/O)
    \end{tikzcd},
    \]
    where the right and outer squares are homotopy pullbacks by definition. We thus have an identification
	\begin{align*}
		\sD &\simeq  \holim \left( \begin{tikzcd} [ampersand replacement=\&] \&  \sL BO \dar{\sL c} \\ BU \rar \& \sL BU \end{tikzcd} \right)\\
		&\simeq  \{ (u,\alpha,\eta) \mid u \in \Map(D^2,BU), \, \alpha \in \sL BO, \, c \circ \alpha \overset{\eta}{\simeq} u|_{\partial D^2} \}. 
	\end{align*}
    Consider the commutative diagram 
    \begin{equation}\label{eq:map_to_u/o}
        \begin{tikzcd}[column sep=1.5cm]
            BU \rar \dar{\Gr} & \sL B(U/O) \dar{\simeq} \\
            B(U/O) \rar{\id_{B(U/O)}\times \ast} & B(U/O) \times U/O
        \end{tikzcd}.
    \end{equation}
    Denote the induced map on the homotopy fibers of the two horizontal maps above by
	\begin{equation}\label{eq:induced_map_loop}
		\rho \colon \sD \longrightarrow \varOmega (U/O)
	\end{equation}
    \begin{rem}\label{rem:index_bundle}
		Every element $(u,\alpha,\eta) \in \sD$ determines a complex vector bundle $u^\ast \sE_\IC$ over the disk and a real sub-bundle of $\alpha^* \sE_\IR$ along the boundary. Let $D_{u,\alpha}$ denote the corresponding Cauchy--Riemann operator acting on the space of sections with boundary conditions in the Lagrangian sub-bundle along $\partial D^2$. This yields the vector bundle
		\[
			\sD \overset{\ind}{\longrightarrow} BO \times \IZ
		\]
		over $\sD$ whose fiber over $(u,\alpha)$ is given by
		\[ \ind(D_{u,\alpha}) = \mathrm{ker}(D_{u,\alpha}) \otimes  (-\mathrm{coker}(D_{u,\alpha})).\]
        The map $\ind$ factors as
        \[ \sD \overset{\rho}{\longrightarrow} \varOmega (U/O) \xrightarrow{\mathrm{Bott}_\IR} BO \times \IZ.
        \]
	\end{rem}
    In analogy with the definition of $B(U/O)^\#$ in \cref{dfn:maslov_data}, define
    \[
    BU^\# \coloneqq \hofib\left(BU \overset{\Gr}{\longrightarrow} B(U/O) \overset{\text{\eqref{eq:maslov_fiber}}}{\longrightarrow} B^2\Pic(R)\right).
    \]
    We denote the induced map $BU^\# \to B(U/O)^\#$ by $\Gr^\#$. 
	\begin{defn}\label{dfn:disk_hash}
        Define
        \[
        \sD^\# \coloneqq \hofib \left(BU^\# \longrightarrow \sL BU^\# \longrightarrow \sL B(U/O)^\#\right).
        \]
	\end{defn}
    There is a diagram similar to \eqref{eq:map_to_u/o} involving  $BU^\#$ and $B(U/O)^\#$ inducing a map
	\[
		\rho^\# \colon \sD^\# \longrightarrow \varOmega (U/O)^\#.
	\]
    \begin{lem}\label{lem:induced_loop_maps_pullback}
        There is a canonical homotopy making the following diagram commutative:\begin{equation}\label{eq:comm_dia_lifts}
		\begin{tikzcd}[sep=scriptsize]
			\sD^\# \dar \rar{\rho^\#} & \varOmega (U/O)^\# \dar \\
			\sD \rar{\rho} & \varOmega (U/O)	
		\end{tikzcd}.
	\end{equation}
    \end{lem}
    \begin{proof}
        The lemma follows by noting that there is a commuting cube
        \[
            \begin{tikzcd}[sep=scriptsize]
                & BU^\# \ar[rr] \ar[dl] \ar[dd] & &\sL B(U/O)^\# \ar[dl] \ar[dd,"\simeq"]\\
                BU \ar[rr,crossing over, near end] \ar[dd] & & \sL B(U/O)\\
                & B(U/O)^\# \ar[rr,near start,"\id \times \ast"] \ar[dl] & & B(U/O)^\# \times (U/O)^\# \ar[dl]\\
                B(U/O) \ar[rr,"\id \times \ast"] & & B(U/O) \times U/O \ar[from=uu,crossing over,near start, "\simeq"]
            \end{tikzcd}
        \]
        and considering the induced maps between the homotopy fibers of the horizontal maps.
    \end{proof}
	\begin{cor}\label{cor:index_R-ori}
	The composition
    \[
		\sD^\# \longrightarrow \sD \overset{\rho}{\longrightarrow} \varOmega(U/O) \xrightarrow{\varOmega^2\mathrm{Bott}_\IR} BO \times \IZ \longrightarrow \Pic(R) 
	\]
	is canonically null-homotopic.
	\end{cor}
	\begin{proof}
		This follows since by \cref{lem:induced_loop_maps_pullback} we have the canonically homotopy commutative diagram
        \[
        \begin{tikzcd}[sep=scriptsize]
			\sD^\# \dar \rar{\rho^\#} & \varOmega (U/O)^\# \dar&& \\
			\sD \rar{\rho} & \varOmega (U/O) \rar{\ind} & BO \times \IZ \rar & \Pic(R),
		\end{tikzcd}
        \]
        and the composition 
		\[
        \varOmega (U/O)^\# \longrightarrow \varOmega (U/O) \longrightarrow \Pic(R) 
        \]
		is canonically null-homotopic by definition.
	\end{proof}
	\subsection{Space of abstract strip caps}\label{sec:space_abstract_strip_caps}
    Define
    \[
        (U/O)^\# \coloneqq \hofib(U/O \xrightarrow{\varOmega\text{\eqref{eq:maslov_fiber}}} B\Pic(R)).
    \]
	\begin{notn}
		For the rest of the section fix $x_0,x_1 \in BO$. We denote their images in $BU^\#,BU,B(U/O)^\#$, and $B(U/O)$ by the same symbols.
	\end{notn}
    It will be useful to fix a subdivision of the disk $D^2$.
    \begin{notn}\label{notn:fixed_nghd}
        \begin{itemize}
            \item Fix a neighborhood $U_{-1} \subset D^2$ of $-1 \in D^2$.
            \item Let $\partial_1 D^2 \coloneqq \overline{U_{-1}} \cap \partial D^2$, and let $\partial_2D^2 \coloneqq (D^2 \smallsetminus U_{-1}) \cap \partial D_2$.
        \end{itemize}
    \end{notn}
    \begin{defn}
		Suppose $\gamma \in \sP_{x_0,x_1} BU$. Then define
		\[
			\sD(\gamma) \coloneqq  \holim \left(
            \begin{tikzcd}[sep=scriptsize]
				& \sP_{x_0,x_1} (\sP_\ast B(U/O)) \dar{\ev_1}\\
				\ast \rar{\Gr\circ \gamma} & \sP_{x_0,x_1}B(U/O)
			\end{tikzcd}
            \right).
		\]
	\end{defn}
    Since $\sP_\ast B(U/O)$ is contractible, there are (non-canonical) weak equivalences
    \[
    \sD(\gamma) \simeq \hofib(\ast \xrightarrow{\Gr \circ \gamma} \sP_{x_0,x_1}B(U/O)) \simeq \varOmega_{\Gr \circ \gamma} \sP_{x_0,x_1} B(U/O) \simeq \varOmega(U/O),
    \]
    which we denote by
    \begin{equation}\label{eq:index_D_gamma}
        \rho \colon \sD(\gamma) \longrightarrow \varOmega(U/O).
    \end{equation}
    Using the commutative diagram,
    \begin{equation}\label{eq:gamma_strips}
        \begin{tikzcd}[sep=scriptsize]
            \sD(\gamma) \rar \dar &  \sP_{x_0,x_1} BO \rar \dar & \sP_{x_0,x_1} (\sP_*(B(U/O))\dar\\
            \ast \rar{\gamma} &  \sP_{x_0,x_1} BU \rar & \sP_{x_0,x_1} B(U/O)
        \end{tikzcd},
    \end{equation}
    where both the squares are pullback squares,
	we have the identification
	\begin{align*}
		\sD(\gamma) &\simeq \holim \left(
			\begin{tikzcd}[sep=scriptsize, ampersand replacement=\&]
				\& \sP_{x_0,x_1} BO \dar{}\\
				\ast \rar{\gamma} \& \sP_{x_0,x_1} BU
			\end{tikzcd}
            \right)\\
            &=\{(u,\alpha) \in \Map(D^2,BU) \times \sP_{x_0,x_1}BO \mid \gamma = u|_{\partial_1 D^2}, \; c \circ \alpha = u|_{\partial_2 D^2}\}.
	\end{align*}
    Via the map $B(U/O) \to B^2\Pic(R)$ defined in \eqref{eq:maslov_fiber}, we consider the following commutative diagram
    \[ 
    \begin{tikzcd}
        \sP_{x_0,x_1} (\sP_\ast B(U/O)) \rar \dar & \sP_{x_0,x_1} B(U/O) \dar\\ 
         \sP_{x_0,x_1} (\sP_\ast B^2\Pic(R) )\rar & 
         \sP_{x_0,x_1} B^2\Pic(R)
    \end{tikzcd}.
    \]
    Taking the homotopy fibers of the horizontal arrows at $\gamma$, we obtain an induced map
    \begin{equation}\label{eq:induced_map_D_pic}
        \sD(\gamma) \longrightarrow \sP_{x_0,x_1} B\Pic(R) \simeq \Pic(R),
    \end{equation}
    where the last identification is not canonical, and that we fix once and for all. Under appropriate identifications of the path and based loop spaces, the map \eqref{eq:induced_map_D_pic} factors as
    \[
    \sD(\gamma) \overset{\rho}{\longrightarrow} \varOmega(U/O) \xrightarrow{\varOmega^2\mathrm{Bott}_\IR} BO \times \IZ \longrightarrow \Pic(\IS) \longrightarrow \Pic(R).
    \]

    Recall that $\sP_\ast X \coloneqq \holim(\ast \to X)$, and for any map $f\colon X \to Y$ there is a canonical map $\sP_\ast X \to \hofib(f)$. In particular, we obtain a canonical map $\sP_\ast B(U/O) \to B(U/O)^\#$.
    \begin{defn}
		For any $\gamma^\# \in \sP_{x_0,x_1}BU^\#$ define
		\[
			\sD^\#(\gamma^\#) \coloneqq \holim
            \left(\begin{tikzcd}[sep=scriptsize]
                & \mathcal P_{x_0,x_1} (\sP_\ast (B(U/O)) \dar{}\\
            \ast \rar{\Gr^\# \circ \gamma^\#}& \mathcal P_{x_0,x_1}B(U/O)^\# 				\end{tikzcd}\right).
		\]
	\end{defn}
    There is a commutative diagram similar to \eqref{eq:gamma_strips},
    \begin{equation}\label{eq:gamma_strips_hash}
        \begin{tikzcd}[sep=scriptsize]
            \sD^\#(\gamma^\#) \rar \dar &  \sP_{x_0,x_1} BO \rar \dar & \sP_{x_0,x_1} (\sP_*(B(U/O))\dar\\
            \ast \rar{\gamma^\#} &  \sP_{x_0,x_1} BU^\# \rar & \sP_{x_0,x_1} B(U/O)^\#
        \end{tikzcd},
    \end{equation}
    where both the squares are pullback squares. Using this we have the identification
	\begin{align*}
		\sD^\#(\gamma^\#) &\simeq \holim \left(
			\begin{tikzcd}[sep=scriptsize, ampersand replacement=\&]
				\& \sP_{x_0,x_1} BO \dar{c^\#}\\
				\ast \rar{\gamma^\#} \& \sP_{x_0,x_1} BU^\#
			\end{tikzcd}
            \right)\\
            &\simeq\left\{(u,\alpha,\eta) \;\middle|\; \begin{matrix}
                u \in \Map(D^2, BU^\#), \; \alpha \in \sP_{x_0,x_1}BO \\
                \gamma^\# = u|_{\partial_1 D^2}, \; c^\# \circ \alpha \overset{\eta}{\simeq} u|_{\partial_2 D^2}
            \end{matrix}\right\}.
	\end{align*}
    For any $\gamma \in \sP_{x_0,x_1}BU$, and a lift $\gamma^\# \in \sP_{x_0,x_1}BU^\#$, we have a commutative diagram
    \[
    \begin{tikzcd}[row sep=scriptsize, column sep=1cm]
        \ast \rar{\Gr^\# \circ \gamma^\#} \dar[equals] & \sP_{x_0,x_1}B(U/O)^\# \dar & \sP_{x_0,x_1}(\sP_\ast B(U/O)) \dar[equals] \lar \\
        \ast \rar{\Gr \circ \gamma} & \sP_{x_0,x_1}B(U/O) & \sP_{x_0,x_1}(\sP_\ast B(U/O)) \lar
    \end{tikzcd},
    \]
    which gives an induced map $\sD^\#(\gamma^\#) \to \sD(\gamma)$.
	\begin{lem}\label{lem:induced_loop_maps_pullback_path}
		For any $\gamma \in \sP_{x_0,x_1}BU$ and a lift $\gamma^\# \in \sP_{x_0,x_1} BU^\#$, the composition
        \[
            \sD^\#(\gamma^\#) \longrightarrow \sD(\gamma) \xrightarrow{\text{\eqref{eq:induced_map_D_pic}}} \Pic(R),
        \]
        is a fiber sequence, and hence canonically null-homotopic.
	\end{lem}
	\begin{proof}
        Consider the commutative diagram
        \[
            \begin{tikzcd}
                \sD^\#(\gamma^\#) \rar \dar & \sP_{x_0,x_1} (\sP_\ast B(U/O)) \rar \dar[equal] & \sP_{x_0,x_1}B(U/O)^\#\dar \\
                \sD(\gamma) \rar \dar & \sP_{x_0,x_1} (\sP_\ast B(U/O)) \rar \dar & \sP_{x_0,x_1} B(U/O) \dar\\
                \sP_{x_0,x_1}B\Pic(R) \rar \dar[sloped,phantom]{\simeq} & \ast \rar & \sP_{x_0,x_1}B^2\Pic(R) \dar[sloped,phantom]{\simeq}\\
                \Pic(R) & & B\Pic(R)
            \end{tikzcd}.
        \]
        The horizontal rows are $\gamma^\#$-homotopy fiber sequences by definition. The canonical null-homotopy is then given by the fact that the middle and right vertical compositions are fiber sequences, implying that the left vertical composition is a fiber sequence.
	\end{proof}
    The following lemma is straightforward and is used in the construction of orientations on flow categories associated to Lagrangians:
	\begin{lem}\label{lem:gluing_disks}
		For any $\gamma \in \sP_{x_0,x_1}BU$ and a lift $\gamma^\# \in \sP_{x_0,x_1} BU^\#$, there are gluing maps such that the following diagram is commutative
		\[
			\begin{tikzcd}[sep=scriptsize]
				\sD^\#(\gamma^\#) \times \sD^\#(\gamma^\#) \rar \dar & \sD^\# \dar \\
				\sD(\gamma) \times \sD(\gamma) \rar& \sD
			\end{tikzcd}.
		\]\qed
	\end{lem}

%% file: SpectralFukaya.tex
\section{Spectral wrapped Donaldson--Fukaya category}\label{sec:spectral_fuk}
    In this section we give a summary of the construction of the wrapped Donaldson--Fukaya category with coefficients in a commutative ring spectrum $R$ as defined in \cite[Section 8]{asplund2024nearby}. The difference between loc.\@ cit.\@ and the present article is that we only assume that $X$ is equipped with a choice of $R$-Maslov data as opposed to a stable polarization. We therefore start with a discussion about orientations.

	We use the Floer theoretic setup as in \cite[Section 6]{asplund2024nearby} based on quadratic Hamiltonians following \cite{sylvan2019on}. We denote by $\sH\sJ(X)$ the space of admissible pairs $(H,J_t)$ of a Hamiltonian $H$ as in \cite[Definition 6.6]{asplund2024nearby} and a domain-dependent almost complex structure $J_t$ as in \cite[Definition 6.4]{asplund2024nearby}. We denote by $\sX(L_0,L_1;H)$ the set of time-$1$ Hamiltonian chords from $L_0$ to $L_1$.
\subsection{Local Cauchy--Riemann operators}
	\begin{notn}
	\begin{enumerate}
		\item Let $R$ be a commutative ring spectrum.
		\item Let $X$ be a Liouville sector equipped with a choice of $R$-Maslov data, i.e., a lift $TX^\# \colon X \to BU^\#$ of the stable tangent bundle of $X$, where $BU^\# = \hofib(BU \to B(U/O) \to B^2\Pic(R))$ (cf.\@ \cref{dfn:maslov_data}).
		\item Let $L_0,L_1 \subset X$ be transversally intersecting Lagrangian $R$-branes (see \cref{eq:r-brane}).
		\item Let $(H_t,J_t) \in \sH\sJ(X)$ such that for each fixed $t$ and each time-$1$ chord $\gamma$ of $H$ from $L_0$ to $L_1$, $J_t$ is constant in a Darboux ball containing $\gamma(t)$.
	\end{enumerate}
\end{notn}

	\begin{defn}
		Given a time-$1$ chord $\gamma$ of $H$ from $L_0$ to $L_1$, the \emph{local operator} at $\gamma$ is the translation invariant operator
		\[ Y_\gamma \in \varOmega^{0,1}_{\IR \times [0,1]} \otimes_J \End(\gamma^\ast TX)\]
		that is determined as follows.  The linearization of the Floer equation at the constant solution $u(s,t) = \gamma(t)$ determines an invertible operator:
		\[ D_\gamma \colon W^{1,2}\left(\IR \times [0,1], (\gamma^\ast TX, L_0|_{\gamma(0)}, L_1|_{\gamma(1)})\right) \longrightarrow L^2\left(\IR \times [0,1], \varOmega^{0,1}_{\IR \times [0,1]} \otimes_J \gamma^\ast TX\right), \]
		where $W^{1,2}\left(\IR \times [0,1], (\gamma^\ast TX, L_0|_{\gamma(0)}, L_1|_{\gamma(1)})\right)$ denotes the space of $W^{1,2}$-sections of $\gamma^\ast TX$ over $\IR \times [0,1]$ which map $\IR \times \{0\}$ to $TL_0|_{\gamma(0)}$ and $\IR \times \{1\}$ to $TL_1|_{\gamma(1)}$.  It is of the form
		\[ D_\gamma(\xi) = \overline{\partial}_{J_t}(\xi) + Y_t(\xi), \]
		where $Y_t = Y_\gamma$.
	\end{defn}
\subsection{Abstract Floer strip caps}\label{sec:floer_caps}
    Let $D_{\pm} \coloneqq D^2 \smallsetminus \left\{\mp 1\right\}\subset \IC$ and fix the standard complex structure on the plane. Recall the fixed neighborhood $U_{-1} \subset D^2$ of $-1$ in \cref{notn:fixed_nghd}. Define
    \[
        \Ends(D_+) \coloneqq \left\{ \varepsilon \colon (-\infty,0] \times [0,1] \hookrightarrow D_+ \mid \varepsilon \text{ is a biholomorphism onto $U_{-1}$.} \right\}.
    \]
    Equip $\Ends(D_+)$ with the subspace topology in the space of smooth maps. Similarly define $\Ends(D_-)$. Notice that $\Ends(D_\pm)$ is contractible.

    Fixing an orientation preserving diffeomorphism $(-\infty,0] \to (-1,0]$ endows $\{-\infty\} \cup (-\infty,0]$ with a smooth topology that it is diffeomorphic to $[-1,0]$. With this, every $\varepsilon \in \Ends(D_+)$ determines a compactification of $D_+$ denoted by $\overline D_+$ and a diffeomorphism
    \begin{equation}\label{eq:diffeo_punc}
        \varphi_\varepsilon \colon \overline D_+ \longrightarrow D^2.
    \end{equation}
    This specifies a decomposition $\partial \overline D_+ = \partial_1 \overline D_+ \cup \partial_2 \overline D_+$, by letting $\partial_i \overline D_+ \coloneqq \varphi_{\varepsilon}^{-1}(\partial_iD^2)$, where $\partial_i D^2$ is defined as in \cref{notn:fixed_nghd}.
        
    Note that given $\varepsilon \in \Ends(D_+)$, the pullback under the diffeomorphism \eqref{eq:diffeo_punc}, and the inclusion $D_+ \hookrightarrow \overline D_+$ specifies a map
    \[
    \Map(D^2,BU) \longrightarrow \Map(D_+,BU).
    \]

    \begin{defn}\label{dfn:abstract_floer_strip_caps}
    Let $\gamma \colon [0,1] \to X$ be a time-$1$ Hamiltonian chord of $H$ from $L_0$ to $L_1$. The set of \emph{positive abstract Floer strip caps} $\sC^\#_+(\gamma)$ associated to $\gamma$ is the set of tuples $((u,\alpha,\eta),\varepsilon,R,J,Y,g)$ where:
    \begin{enumerate}
    	\item $(u,\alpha,\eta) \in \sD^\#(\gamma^\ast TX^\#)$.
        \item $\varepsilon \in \Ends(D_+)$.
    	\item $R \in \IR_{>0}$.
    	\item $J$ is an almost complex structure on $u^\ast TX$ over $D_+$ such that $(\varepsilon|_{(-\infty,-R) \times [0,1]})^\ast J$ agrees with the given almost complex structure $J_t$ along $\gamma$.
    	\item $Y \in \varOmega^{0,1}_{D_+} \otimes_J \End((u^\ast TX)^\ast \sE)$ such that $Y_\gamma = (\varepsilon|_{(-\infty,-R) \times [0,1]})^\ast Y$.
    	\item $g$ is a metric on $D_+$ such that $(\varepsilon|_{(-\infty,-R) \times [0,1]})^\ast g$ agrees with the standard metric on $\IR \times [0,1] \subset \IR^2$.
    
    \end{enumerate}
    We endow $\sC^\#_+(\gamma)$ with the subspace topology of the product topology. We define $\sC_-^\#(\gamma)$ analogously.
    \end{defn}
    \begin{rem}
        If $u \colon D^2\to X$ is a smooth map with Lagrangian boundary conditions, the space $\sC^\#_+(\gamma)$ is non-empty assuming that $X$ and the involved Lagrangian admit $R$-Maslov data. This is important when constructing $R$-orientations on flow categories coming from Floer theory in \cref{sec:category}.
    \end{rem}
    We also define spaces $\sC_{\pm}(\gamma)$ by using $\sD(\gamma^\ast TX)$ instead of $\sD^\#(\gamma^\ast TX^\#)$ in \cref{dfn:abstract_floer_strip_caps}. Via the map $\pi \colon \sD^\#(\gamma^\ast TX^\#) \to \sD(\gamma^\ast TX)$ defined in \cref{lem:induced_loop_maps_pullback_path}, we obtain projection maps
    \begin{align*}
    	p_\pm \colon \sC_{\pm}^\#(\gamma) &\longrightarrow \sC_{\pm}(\gamma) \\
    	((u,\alpha,\eta),\varepsilon,R,J,Y,g) &\longmapsto (\pi(u,\alpha,\eta),\varepsilon,R,J,Y,g).
    \end{align*}
    
    \begin{lem}\label{lma:forgetful_maps_caps}
    There are forgetful maps
    
    \begin{center}
    	\begin{minipage}{0.3\textwidth}
    	\begin{align*}
    		\sC^\#_{\pm}(\gamma) &\longrightarrow \sD^\#(\gamma^\ast TX^\#) \\
    		((u,\alpha,\eta),\varepsilon,R,J,Y,g) &\longmapsto (u,\alpha,\eta)
    	\end{align*}
    \end{minipage}
    \begin{minipage}{0.3\textwidth}
    	\begin{align*}
    		\sC_{\pm}(\gamma) &\longrightarrow \sD(\gamma^\ast TX) \\
    		((u,\alpha,\eta),\varepsilon,R,J,Y,g) &\longmapsto (u,\alpha,\eta)
    	\end{align*}
    \end{minipage}
    
    \end{center}
    that are continuous homotopy equivalences which moreover commute with the projections $p_\pm$. 
    \end{lem}
    
    \begin{proof}
    Continuity and the condition that they commute with the projections $p_\pm$ are clear.  The statement about homotopy equivalence follows from the fact that the set of tuples $(\varepsilon,R,J,Y,g)$ satisfying items (ii)--(vi) in \cref{dfn:abstract_floer_strip_caps} is contractible.
    \end{proof}
    \subsection{Index bundles of abstract strip caps}\label{sec:index_bundle_abstract_strip_caps}
    Consider an element
    \[
    ((u,\alpha,\eta),\varepsilon,R,J,Y,g) \in \sC_{\pm}^\#(\gamma^\ast TX^\#).
    \]
    The triple $((u,\alpha,\eta),\varepsilon)$ determines the linear operator
    \begin{align*}
        D_v \colon W^{1,2}\left(\overline{D}_\pm, ((u^\ast TX)^\ast \sE, \alpha)\right) &\longrightarrow L^2\left(\overline{D}_\pm, \varOmega^{0,1}_{\overline D_\pm} \otimes_J (u^\ast TX)^\ast \sE \right) \\
        \xi &\longmapsto \overline{\partial}_J(\xi) + Y(\xi).
    \end{align*}
    \begin{defn}\label{dfn:index_bundle_caps}
    	For any Hamiltonian chord $\gamma \colon [0,1] \to X$ with endpoints on $L_0$ and $L_1$, respectively, the \emph{index bundle} of $\sC_{\pm}^\#(\gamma)$ is the vector bundle classified by the composition
    	\[
    		\sV_\pm(\gamma) \colon \sC_\pm^\#(\gamma) \longrightarrow \sD^\#(\gamma^\ast TX^\#) \overset{\rho}{\longrightarrow} \varOmega(U/O) \xrightarrow{\varOmega^2\mathrm{Bott}_\IR} BO \times \IZ,
    	\]
    	i.e., it is the vector bundle $\sV_\pm(\gamma) \to \sC_{\pm}^\#(\gamma)$ with fiber over $((u,\alpha,\eta),\varepsilon,R,J,Y,g)$ given by $\ind(D_v)$ (cf.\@ \cref{rem:index_bundle}).
    \end{defn}
    The following is a direct consequence of \cref{lem:induced_loop_maps_pullback_path}.
    \begin{cor}
    	For any Hamiltonian chord $\gamma \colon [0,1] \to X$ with endpoints on $L_0$ and $L_1$, respectively, the index bundle $\sV_\pm(\gamma)$ is canonically $R$-orientable.
        \qed
    \end{cor}
\subsection{The spectral wrapped Donaldson--Fukaya category}\label{sec:category}
    We now give a brief description of the construction of the wrapped Donaldson--Fukaya category from \cite[Section 8]{asplund2024nearby}.

\begin{defn}\label{dfn:floer_strips_boudary_asymp}
	Given $a,b \in \sX(L_0,L_1)$, define $\widetilde{\sR}(a,b;H,J_t)$ to be the set of maps 
	\[
		\widetilde{\sR}(a;b;H,J_t) \coloneqq \left\{ u\colon \IR \times [0,1] \to \overline{X} \,\left|\, \begin{matrix}
	\partial_s u + J_{t}\left(\partial_t u - X_{H}\right) = 0 \\
	u(\IR \times \left\{0\right\}) \subset L_0 \\
	u(\IR \times \left\{1\right\}) \subset L_1 \\
	\lim_{s\to -\infty} u(s,t) = b(t) \\
	\lim_{s\to \infty} u(s,t) = a(t) \end{matrix} \right. \right\}.
	\]
\end{defn}
There is a natural $\IR$-action given by translation of the $s$-coordinate in the domain. Define $\sR(a;b;H,J_t) \coloneqq \widetilde{\sR}(a;b;H,J_t)/\IR$ whenever this action is free and declare it to be empty otherwise.
\begin{defn}
	Define $\osr(a;b;H,J_t)$ to be the Gromov compactification of $\sR(a;b;H,J_t)$ (see \cite[(9l)]{seidel2008fukaya}).
\end{defn}
In the following, $\mu$ denotes the Maslov index, see \cite{robbin1993maslov}.
\begin{lem}\label{lem:transversality_for_strips}
There exists a comeager set of $(H,J_t) \in \sH\sJ(\overline{X},F)$ such that the moduli space $\osr(a;b)$ is a smooth manifold with corners of dimension $\mu(b)-\mu(a)-1$.  Moreover, the following conditions are satisfied:
\begin{enumerate}
		\item For $a,b,c \in \sX(L_0,L_1;H)$, there is a map
		\[ \mu_{abc}\colon \osr(a;b) \times \osr(b;c) \longrightarrow \osr(a;c) \]
		that is a diffeomorphism onto a face of $\osr(a;c)$ such that defining
		\[ \partial_i \osr(a;c) \coloneqq \bigsqcup_{\substack{b \\ \mu(b)-\mu(a) - 1 = i}} \osr(a;b) \times \osr(b;c),\]
		endows $\osr(a;c)$ with the structure of a $\ang{\mu(b)-\mu(a) - 1}$-manifold (see \cite[Section 2.2]{asplund2024nearby}).
		\item There exist rank one free $R$-modules $\check{\fo}(a)$ for all $a \in \sX(L_0,L_1;H)$ and isomorphisms:
		\[ \psi_{ab} \colon \check{\fo}(b) \overset{\cong}{\longrightarrow} \check{\fo}(a)  \otimes_RI_R(a;b)\]
		that are compatible with the maps $\mu_{abc}$.
\end{enumerate}
\end{lem}
\begin{proof}
    \begin{enumerate}
        \item The existence of a smooth structure respecting the corner structure on $\osr(a;b)$ is due to Large \cite[Section 6]{large2021spectral} and Fukaya--Oh--Ohta--Ono \cite{fukaya2016exponential}.
        \item This is completely analogous to the proof of \cite[Lemma 6.18]{asplund2024nearby}, see \cite[Lemma 7.8]{asplund2024nearby}.
    \end{enumerate}
\end{proof}
\begin{defn}[Flow category associated to Lagrangians]\label{dfn:flow_cat_lag}
	Let $L_0, L_1 \subset X$ be two Lagrangian $R$-branes and let $(H,J_t) \in \sH\sJ(X)$ be so that \cref{lem:transversality_for_strips} holds. The \emph{flow category associated to $L_0$, $L_1$, and $(H,J_t) \in \sH\sJ(\overline X,F)$} is denoted by $\sC\sW(L_0,L_1;H,J_t)$ and is defined as follows:
    \begin{itemize}
        \item The set of objects is given by $\Ob(\sC\sW(L_0,L_1;H,J_t)) \coloneqq \mathcal X(L_0,L_1,H)$.
        \item The grading function is given by $-\mu$ (i.e.\@ \emph{negative} of the Maslov index).
        \item The morphisms from $a$ to $b$ are given by $\sC\sW(L_0,L_1;H,J_t)(a;b) \coloneqq \overline{\sR}(a;b;H,J_t)$.
        \item The $R$-orientation is given by $\fo_{(L_0,L_1;H)}$ is given by $\fo(a) \coloneqq (-\sV(a))_R$ and the isomorphism of $R$-line bundles
        \[
            \psi_{ab} \colon \fo(a) \overset{\simeq}{\longrightarrow} I_R(x;y) \otimes_R \fo(b),
        \]
        given by \cref{lem:transversality_for_strips}(ii). Here $\sV(a) = \sV_+(a)$ is as in \cref{dfn:index_bundle_caps}, and $(-\sV(a))_R$ denotes by abuse of notation the generic fiber of the associated $R$-line bundle $(-\sV(a))_R$, 
    \end{itemize}
\end{defn}
\begin{notn}
	\begin{enumerate}
		\item Letting $c \in \sX(L_0,L_1;H)$ we denote the action functional of $c$ associated to Floer's equation by
		\[
				\sA(c) \coloneqq \int c^\ast \lambda - \int H(c(t)) dt - (h_{1}(c(1)) - h_{0}(c(0))),
		\]
		where $h_0 \colon L_0 \to \IR$ and $h_1 \colon L_1 \to \IR$ are smooth functions satisfying $\lambda|_{L_0} = dh_0$ and $\lambda|_{L_1} = dh_1$, respectively.
		\item For any $A \in \IR_{> 0}$, set
		\[
				\sX^{\leq A}(L_0,L_1;H) \coloneqq \left\{c \in \sX(L_0,L_1;H) \mid \sA(c) \leq A\right\}.
		\]
	\end{enumerate}
\end{notn}
\begin{defn}[Action filtered flow category associated to Lagrangians]\label{dfn:action_filtered_flow_cat}
	Let $A \in \IR_{> 0}$ and define $\sC\sW^{\leq A}(L_0,L_1;H,J_t)$ to be the flow category that is defined in the same way as $\sC\sW(L_0,L_1;H,J_t)$ (see \cref{dfn:flow_cat_lag}) except that the set of objects is
	\[
	    \Ob(\sC\sW^{\leq A}(L_0,L_1;H,J_t)) \coloneqq \sX^{\leq A}(L_0,L_1;H).
	\]
	We denote the corresponding $R$-orientation by $\fo^{\leq A_k}_{(L_0,L_1)}$.
\end{defn}
\begin{notn}
    For $A > 0$, let
    \[
        HW^{\leq A}(L_0,_1;H,J_t) \coloneqq |\sC\sW^{\leq A}(L_0,L_1;H,J_t),\fo_{(L_0,L_1)}^{\leq A}|,
    \]
    where the right hand side denotes the CJS realization, see \cite[Definition 3.27]{asplund2024nearby}. 
\end{notn}
\begin{defn}
    We define
	\[
		HW(L_0,L_1;H,J_t) \coloneqq \hocolim_{k \to \infty} HW^{\leq A_k}(L_0,L_1;H,J_t),
	\]
    where $\{A_k\}_{k=1}^\infty \subset \IR$ denote an increasing sequence diverging to $\infty$.	
\end{defn}

\begin{rem}
	\begin{enumerate}
        \item $HW(L_0,L_1;H,J_t)$ is independent of the choice of the subsequence $\{A_k\}_{k=1}^\infty \subset \IR$.
		\item It was shown in \cite[Lemma 8.11]{asplund2024nearby} that there is an equivalence of $R$-modules
		\[
			HW(L_0,L_1;H^1,J_t^1) \simeq HW(L_0,L_1;H^2,J_t^2),
		\]
		for two different choices of regular pairs $(H^1,J_t^1),(H^2,J_t^2) \in \sH\sJ(X)$. Therefore by abuse of notation we write
		\[
			HW(L_0,L_1) \coloneqq HW(L_0,L_1;H,J_t).
		\]
		\item In the special case of $R = H\IZ$, the homotopy groups of the $H\IZ$-module $HW(L_0,L_1;H\IZ)$ are isomorphic to the wrapped Floer cohomology $HW^{-\bullet}(L_0,L_1;\IZ)$ with reversed grading, see \cite[Lemma 8.8]{asplund2024nearby}. The proof in the current setting (without the existence of a stable polarization) is the same, except that the comparison of signs and orientations is proven in \cref{lem:hz_maslov}.
	\end{enumerate}
\end{rem}
The triangle product in classical Floer theory admits a spectral lift that is a homotopy associative product
\begin{equation}\label{eq:triangle_product}
	\mu^2 \colon HW(L_0,L_1) \wedge_R HW(L_1,L_2) \longrightarrow HW(L_0,L_2),
\end{equation}
see \cite[Lemma 8.15]{asplund2024nearby}. There is moreover a (Floer) unit map
\begin{equation}\label{eq:unit_hw}
	\eta_L \colon R \longrightarrow HW(L,L),
\end{equation}
which is a unit in the sense that the following composition is homotopic to the identity:
\[
	HW(L,K) \overset{\simeq}{\longrightarrow} R \wedge_R HW(L,K) \overset{\eta_L \wedge_R \id}{\longrightarrow} HW(L,L) \wedge_R HW(L,K) \overset{\mu^2}{\longrightarrow} HW(L,K),
\]
see \cite[Lemma 8.17]{asplund2024nearby}.

\begin{defn}[Spectral wrapped Donaldson--Fukaya category]\label{dfn:spectral_df_category}
	The \emph{wrapped Donaldson--Fukaya category with coefficients in $R$} of $X$ is the category $\sW(X;R)$ enriched over the homotopy category of $R$-modules whose objects are given by the set of Lagrangian $R$-branes in $X$, the morphisms from $L_0$ to $L_1$ are given by
	\[ \sW(\overline X,F;R)(L_0,L_1) \coloneqq HW(L_0,L_1), \]
	and compositions of morphisms are given by $\mu^2$ in \eqref{eq:triangle_product}.
\end{defn}

%% file: OC.tex
\subsection{The open-closed map}
	\begin{lem}[Open-closed map]
		Let $L \subset X$ be an $R$-orientable Lagrangian $R$-brane. There exists a map of $R$-modules
		\[
			\sO\sC \colon HW(L,L) \longrightarrow \varSigma^{\infty-n}_+X \wedge R,
		\]
		with the property that the composition
		\[
			R \overset{\eta_L}{\longrightarrow} HW(L,L) \overset{\sO\sC}{\longrightarrow} \varSigma^{\infty-n}_+ X \wedge R,
		\]
		coincides with the $R$-fundamental class $[L]_R \in H_n(X;R)$.
        \qed
	\end{lem}
	\begin{proof}
		The open-closed map $\sO\sC$ is defined as the CJS realization of the flow bimodule constructed in \cite[Definition 9.9]{asplund2024nearby}, and the latter claim is proven in \cite[Lemma 9.12]{asplund2024nearby}.
	\end{proof}
    The following result is analogous to \cite[Theorem 9.14]{asplund2024nearby} in the present setting without the existence of a stable polarization on $X$.
	\begin{thm}\label{thm:oc_fund_classes}
		Let $L$ and $K$ be two closed Lagrangian $R$-branes such that $L \cong K$ in $\sW(X;R)$.
        \begin{enumerate}
            \item If $L$ is $R$-orientable, then $K$ is $R$-orientable.
            \item Given an $R$-orientation on $L$, there exists an $R$-orientation on $K$ such that $[L]_R = [K]_R \in H_n(X;R)$.\qed
        \end{enumerate}
	\end{thm}

%% file: NearbyCocoreEquivalent.tex
\section{Spectral equivalence of nearby Lagrangian cocores}\label{sec:nearby_cocore}
	The goal of this section is to prove that, in the spectral wrapped Donaldson--Fukaya category, a nearby Lagrangian cocore (defined in \cref{dfn:nearby_cocore}) is equivalent to the cocore it is near after attaching a standard Weinstein handle along the boundary of the cocore. 

	\subsection{Invariance under subcritical handle attachment}\label{subsec:subcrit_handles}
		Recall that a Weinstein sector may be viewed as the result of successive attachments of Weinstein handles and Weinstein half-handles, see \cref{sec:geom_background}. Working over a discrete coefficient ring $k$, the wrapped Fukaya category of a Weinstein sector is invariant up to quasi-equivalence when attaching subcritical Weinstein handles or Weinstein half-handles, see \cite[Corollary 1.29]{ganatra2022sectorial}. We now extend this result to the spectral Donaldson--Fukaya category.

		\begin{notn}
            \begin{enumerate}
                \item Let $R$ denote a commutative ring spectrum. 
                \item Let $X$ denote a Weinstein sector admitting $R$-Maslov data.
                \item We use the notation $HW(-,-)_X$ to emphasize that the wrapped Floer homotopy $R$-module is computed in $X$.
            \end{enumerate}
		\end{notn}
        Recall the definition of a (trivial) inclusion of Liouville sectors $i \colon X \hookrightarrow X'$ from \cite[Definition 2.4]{ganatra2020covariantly}. By \cite[Theorem 8.21]{asplund2024nearby} any inclusion of Liouville sectors induces a pushforward functor $\sW(X;R) \to \sW(X';R)$. We now show that the pushforward functor is a weak equivalence in case the inclusion of Liouville sectors is trivial. Recall the definition of the $R$-oriented flow bimodules $\osr_L \colon \sM_\ast \to \sC\sW^{\leq 0}(L,L)$ from \cite[Section 8.3]{asplund2024nearby} and $\osr^\tau \colon \sC\sW^{\leq A_k}(L,L;H^0,J_t^0) \to \sC\sW^{\leq A_k}(L,L;H^1,J_t^1)$ from the proof of \cite[Lemma 8.11]{asplund2024nearby}. Suppose that $(H^r,J_t^r) \subset \sH\sJ(\overline X,F)$ is an admissible $1$-parameter family of Floer data such that $\sC\sW^{\leq A_k}(L,L;H^1,J_t^1) \simeq \sC\sW^{\leq A_k}(L',L;H^0,J_t^0)$, where $L'$ is an appropriate Hamiltonian push-off of $L$ that depends on $H^1$.
        \begin{defn}[Continuation flow bimodule]\label{dfn:cont_flow_bimod}
            Let $L$ and $L'$ be two Hamiltonian isotopic Lagrangian $R$-branes. Define the \emph{continuation flow bimodule} $\overline{\sC}_{L',L}$ as the composition of the following $R$-oriented flow bimodules:
            \[
            \sM_\ast \overset{\osr_L}{\longrightarrow} \sC\sW^{\leq 0}(L,L) \overset{\osr^\tau}{\longrightarrow} \sC\sW^{\leq 0}(L',L).
            \]
        \end{defn}
        \begin{lem}\label{lem:prod_continuation}
            Let $L$, $L'$, and $K$ be Lagrangian $R$-branes, and suppose that $L$ and $L'$ are Hamiltonian isotopic. The following composition is homotopic to $|\osr^\tau|$:
            \begin{align*}
                HW(L,K) \simeq R \wedge_R HW(L,K) &\xrightarrow{|\overline{\sC}_{L',L}| \wedge_{R} \id} HW(L',L) \wedge_R HW(L,K) \\
                &\overset{\mu^2}{\longrightarrow} HW(L',K).
            \end{align*}
        \end{lem}
        \begin{proof}[Proof sketch]
            This is analogous to the proof of \cite[Lemma 8.17]{asplund2024nearby}. The composition is given as the CJS realization of the $R$-oriented flow multimodule
            \[
            \osr^k \circ_1 \overline{\sC}_{L',L} \colon \sM_\ast,\sC\sW^{\leq A_k}(L,K) \longrightarrow \sC\sW^{\leq A_k}(L',K),
            \]
            and its associated $R$-oriented flow bimodule $\sU^k$ defined by $\sU^k \coloneqq (\osr^k \circ_1 \overline{\sC}_{L',L})(p,-)$. We construct an $R$-oriented flow bordism $\sU^k \Rightarrow \osr^\tau$, similar to the one defined by \cite[Lemma 6.46]{asplund2024nearby}, and taking CJS realizations \cite[Proposition 3.32]{asplund2024nearby} finishes the proof.
        \end{proof}
        \begin{lem}\label{lem:trivial_incl_quasi_iso}
            Let $X \hookrightarrow X'$ be a trivial inclusion of Liouville sectors, then the pushforward functors $\sW(X;R) \to \sW(X';R)$ and $\sW^{\leq A}(X;R) \to \sW^{\leq A}(X';R)$ for any $A > 0$ are weak equivalences.
        \end{lem}
        \begin{proof}[Proof sketch]
            The idea of the proof is the same as the standard proof, see e.g.\@ \cite[Lemma 3.33]{ganatra2020covariantly} and \cite[Proposition 3.15(4)]{sylvan2019on}. Fix Lagrangian $R$-branes $L$ and $K$ in $X$. The geometric parts from the proof of \cite[Lemma 3.33]{ganatra2020covariantly} carries over verbatim. Namely, it suffices to consider sequences of ``small inward shrinkings'' of $X'$, and the argument in the proof of \cite[Lemma 3.33]{ganatra2020covariantly} shows that the induced pushforward functor on morphisms is given by the composition in \cref{lem:prod_continuation}, which is shown to be homotopic to the weak equivalence $|\osr^\tau|$.
            
            The fact that the continuation flow bimodule defined in \cref{dfn:cont_flow_bimod} is supported in non-positive action shows that the action filtered pushforward for any $A > 0$, is a weak equivalence.
        \end{proof}
       
        \begin{lem}\label{lem:invariance_subcrit}
		Let $X^{2n}$ be a Weinstein sector admitting $R$-Maslov data, and let $X_+$ be the result of attaching a subcritical Weinstein (half-)handle to $X$. The pushforward functor $\sW(X;R) \hookrightarrow \sW(X_+;R)$ is a weak equivalence.
    	\end{lem}
	    \begin{proof}
            Fix two Lagrangian $R$-branes $L$ and $K$ in $X$. Fix some Floer datum so that $HW(L,K)_X$ is well-defined. Let $\varepsilon > 0$ and denote by $X_+(\varepsilon)$ the Weinstein sector obtained by attaching a Weinstein (half-)handle $H_\varepsilon$ of size $\varepsilon$ (see \cref{dfn:weinstein_handle,dfn:weinstein_half-handle}). By convention we fix some $\varepsilon_0 > 0$ and let $X_+ \coloneqq X_+(\varepsilon_0)$. The isotropic attaching sphere of $H_\varepsilon$ has dimension less than $n-1$, and therefore we may without loss of generality assume that no Hamiltonian chord of $F$ of action bounded above by $A$ intersects the Legendrian attaching sphere of $H_\varepsilon$ (otherwise take an arbitrarily small compactly supported Hamiltonian perturbation of $F$ to achieve it). Therefore, for any given $A > 0$, we may find some small enough $\varepsilon_A > 0$ (depending on $A$) such that 
	    	\begin{equation}\label{eq:ob_ident}
	    		\Ob(\sC\sW^{\leq A}(L,K)_X) = \Ob(\sC\sW^{\leq A}(L,K)_{X_+(\varepsilon_A)}).
	    	\end{equation}
	    	There is an inclusion of Weinstein sectors $X \hookrightarrow X_+(\varepsilon)$ for any $\varepsilon > 0$. By the discussion in \cite[Section 3.6]{ganatra2020covariantly} (cf.\@ the proof of \cite[Theorem 8.21]{asplund2024nearby}) we get an identification of morphism spaces
	    	\[
	    		\sC\sW(L,K)_X(x;y) \cong \sC\sW(L,K)_{X_+(\varepsilon)}(x;y),
	    	\]
	    	for any $x,y\in \Ob(\sC\sW^{\leq A}(L,K)_X)$. Combining this with \eqref{eq:ob_ident} yields an identification of $R$-oriented flow categories 
	    	\begin{equation}\label{eq:equiv_filtered_flow}
	    		\sC\sW^{\leq A}(L,K)_X \cong \sC\sW^{\leq A}(L,K)_{X_+(\varepsilon_A)}.
	    	\end{equation}
            It follows that the pushforward functor $\sW^{\leq A}(X;R) \hookrightarrow \sW^{\leq A}(X_+(\varepsilon_A);R)$ is a weak equivalence. Now, pick a strictly increasing sequence $\left\{A_k\right\}_{k=1}^\infty$ of real numbers such that $A_k \to \infty$ as $k \to \infty$, and a sequence $\left\{\varepsilon_k\right\}_{k=1}^\infty$ such that \eqref{eq:ob_ident} holds for each pair $(A_k,\varepsilon_k)$. For any $\varepsilon < \varepsilon'$, there is a trivial inclusion of Weinstein sectors $X_+(\varepsilon) \hookrightarrow X_+(\varepsilon')$, which means that the pushforward is a weak equivalence $\sW(X_+(\varepsilon);R) \simeq \sW(X_+(\varepsilon');R)$. Therefore we obtain
            \[
            \sW(X;R) \simeq \hocolim_{k\to \infty} \sW^{\leq A_k}(X_+(\varepsilon_{A_k});R) \simeq \hocolim_{k\to \infty} \sW^{\leq A_k}(X_+;R) \simeq \sW(X_+;R).
            \]
	    \end{proof}
        
	\subsection{Nearby Lagrangian cocores}\label{sec:nearby_lag_cocore}
		Fix a Weinstein handlebody decomposition of the Weinstein sector $X$ (see \cref{dfn:handlebody_decomp_sector}). Let $\left\{C_i\right\}_i$ denote the set of Lagrangian cocore disks (including those of the critical Weinstein half-handles), and $\boldsymbol C \coloneqq \bigcup_{i} C_i$ their union. The result after removing all critical Weinstein handles from $X$ is called the \emph{subcritical part} of $X$, and is denoted by $X_0$.
		\begin{defn}[Shifted Lagrangian cocore]\label{dfn:shift_cocore}
            \hfill
			\begin{itemize}
			    \item In the local model of a critical Weinstein handle (see \cref{dfn:weinstein_handle}), a \emph{shifted Lagrangian cocore} is the subset $(\IR^n \times \{\varepsilon\}) \cap H^\delta_n \subset H^\delta_n$ for some $\varepsilon \in \IR^n \smallsetminus \left\{0\right\}$ close enough to $0$.

                \item In the local model of a critical Weinstein half-handle (see \cref{dfn:weinstein_half-handle}), the symplectic boundary is a Weinstein handle of index $n-1$. Define a \emph{shifted Lagrangian cocore} of a critical Weinstein half-handle to be the spreading of a shifted cocore of its symplectic boundary (see \cite[Definition 2.11]{chantraine2017geometric} for the definition of spreading).
			\end{itemize}
		\end{defn}
        \begin{rem}
            The Lagrangian cocore of a Weinstein half-handle corresponds to a linking disk, in the language of Ganatra--Pardon--Shende \cite{ganatra2020covariantly,ganatra2022sectorial}, see \cref{rem:cocore_linking_disk}.
        \end{rem}
		\begin{defn}[Nearby Lagrangian cocore]\label{dfn:nearby_cocore}
			We say that an exact conical Lagrangian $L \subset X$ is a \emph{nearby Lagrangian cocore} if $\partial_\infty L = \partial_\infty C_\varepsilon$ where $C_\varepsilon$ denotes a shifted Lagrangian cocore of $X$, and if $L$ misses every Lagrangian cocore of $X$.
		\end{defn}
        Nearby Lagrangian cocores are closely related to exact Lagrangian fillings of Legendrian unknots in subcritical Weinstein sectors.
		\begin{lem}\label{lem:nearby_filling}
			Items of the following two kinds can be used to produce one of the other:
			\begin{enumerate}
				\item A Weinstein sector $X$ together with a nearby Lagrangian cocore $L \subset X$.
				\item A subcritical Weinstein sector $X_0$ together with an exact Lagrangian filling of a Legendrian unknot $\varLambda_0$ in a Darboux ball in $\partial_\infty X_0$.
			\end{enumerate}
		\end{lem}
		\begin{proof}
			\begin{description}
				\item[(i) $\Rightarrow$ (ii)] This follows from the observation that the backwards Liouville flow will eventually disjoin $L$ from each critical Weinstein (half-)handle, since $L$ misses every Lagrangian cocore disk. Let $\varLambda \subset \partial_\infty X_0$ denote the Legendrian attaching sphere of the Weinstein handle housing the shifted cocore $C_\varepsilon$. In the local model of a critical Weinstein handle in \cref{dfn:weinstein_handle}, the time-$t$ backwards Liouville flow is given by $\varphi^t(x,y) = (e^{-2t}x,e^ty)$, and the contact boundary of the shifted cocore is given by the subset
                \[
                \partial_\infty C_\varepsilon = \left\{(x,\varepsilon) \in \IR^n \times \IR^n \; \middle| \; \sum_{i=1}^n x_i^2 - \frac 12 \varepsilon_i^2 = \delta\right\} \subset \partial_+ H^\delta_n,
                \]
                for some $\varepsilon \in \IR^n$ close enough to the origin. The projection to the first $n$ coordinates is a sphere of small radius. Let $T > 0$ be such that 
                \[
                \sum_{i=1}^ne^{-4T}x_i^2 - \frac 12 e^{2T}\varepsilon_i^2 = -\delta,
                \]
                i.e., such $\varphi^T(\partial_\infty C_\varepsilon) \subset \partial_- H^\delta_n$; note that for $\sum_{i=1}^n \varepsilon_i^2 \ll \delta$, the projection to the $n$ first coordinates is a sphere of smaller radius than the radius of the shifted cocore sphere. The attaching sphere for the critical Weinstein handle $H^\delta_n$ is given by the subset 
                \[
                A \coloneqq \left\{(0,y) \in \IR^n \times \IR^n \; \middle| \; \sum_{i=1}^n y_i^2 = 2\delta\right\} \subset \partial_- H^\delta_n,
                \]
                and we therefore see that $\varphi^T(\partial_\infty C_\varepsilon)$ is a small sphere linking the attaching sphere near the point $\left(0,\sqrt{\frac{2\delta}{\sum_{i=1}^n \varepsilon_i^2}}\varepsilon\right) \in A$. Therefore $\varphi^T(L)$ for $T > 0$ large enough is completely contained in the subcritical part of $X$ and is by construction an exact Lagrangian filling of its Legendrian boundary $\varphi^T(\partial_\infty L) = \varphi^T(\partial_\infty C_\varepsilon) \subset \partial_\infty X_0$.
				\item[(ii) $\Rightarrow$ (i)] Choose a Legendrian sphere $\varLambda \subset \partial_\infty X_0$ that links with $\varLambda_0$ exactly once. Attach a critical Weinstein handle to $X_0$ along $\varLambda$ and additional critical Weinstein (half-)handles to $X_0$ with attaching locus disjoint from the Darboux ball $D$.
			\end{description}
		\end{proof}
		Since nearby Lagrangian cocores can be used to construct exact Lagrangian fillings of the Legendrian unknot, it follows that there are considerable topological restrictions, cf.\@ \cite[Theorems 1.9 and 1.13]{chantraine2020floer}.
		\begin{lem}\label{lma:nearby_cocore_is_a_disk}
			Assume $2c_1(X) = 0$. Any nearby Lagrangian cocore $L \subset X$ is diffeomorphic to a disk if $n \neq 4$, and homeomorphic to a disk if $n = 4$.
		\end{lem}
		\begin{proof}
			By \cref{lem:nearby_filling} both $L$ and $C_\varepsilon$ produce exact Lagrangian fillings of a Legendrian unknot $\varLambda_0$ in a Darboux ball in $\partial_\infty X_0$ that we by abuse of notation denote by $L,C_\varepsilon \subset X_0$.

            First, using (ungraded) wrapped Floer cohomology with $\IZ/2$-coefficients we have isomorphisms $HW^\bullet(L;\IZ/2) \cong HW^\bullet(C_\varepsilon;\IZ/2) \cong 0$, which implies $H^\bullet(L;\IZ/2) \cong H^\bullet(C_\varepsilon;\IZ/2) \cong \IZ/2$ generated by the unit, as ungraded $\IZ/2$-vector spaces, since $C_\varepsilon$ is an $n$-dimensional disk. This in particular implies that $L$ is spin. This allows us to repeat the argument with integral coefficients to conclude $H^\bullet(L;\IZ) \cong H^\bullet(C_\varepsilon;\IZ) \cong \IZ$ generated by the unit as ungraded $\IZ$-modules, and hence $L$ is an integer homology disk. 
            
            Next, the Chekanov--Eliashberg dg-algebra of $\varLambda_0 \subset \partial_\infty X_0$ is quasi-isomorphic to that of the standard Legendrian unknot in $\IR^{2n-1}$ (see \cite[Lemma 4.1]{asplund2023singular}) which is generated by a single Reeb chord $c$ in degree $-(n-1) < 0$. The above observation about the cohomology of $L$ implies that the Maslov class of $L$, vanishes, and inspection of the proof of \cite[Theorem 70]{ekholm2017duality} shows that it is still applicable under our slightly weaker assumption, viz., with \@ $2c_1(X)=0$ instead of $c_1(X)=0$. Therefore we conclude that $L$ is a simply connected homology disk, and hence $L$ is a homotopy disk. For $n \geq 5$, arguments using the $h$-cobordism theorem can then be employed to show that it is diffeomorphic to a disk (see \cite[Section 9]{milnor2015lectures}). For $n = 3$, we can attach a critical Weinstein handle to $X$ along $\partial_\infty L$ to construct a homotopy sphere which therefore must be diffeomorphic to $S^3$. 
		\end{proof}
        \begin{rem}
            We highlight here that \cref{lma:nearby_cocore_is_a_disk} implies that if a nearby Lagrangian $L \subset T^\ast S^n$ has geometric intersection number equal to one with some cotangent fiber, it must be diffeomorphic to $S^n$ if $n\neq 4$ (cf.\@ \cite[Corollary 1.19]{lazarev2020contact}).
            
            Viz., we may assume without loss of generality that $L$ coincides with the zero section over a small closed disk in the zero section. Pick a Weinstein handlebody decomposition of $T^\ast S^n$ with one index $0$ Weinstein handle and one critical Weinstein handle attached to the standard Legendrian unknot in the standard contact $S^{2n-1}$, such that $L$ coincides with the zero section over the core disk of the critical Weinstein handle. Restricting $L$ and the zero section to the index $0$ Weinstein handle $B^{2n}$ yields two exact Lagrangian fillings of the standard Legendrian unknot. Repeating the proof of \cref{lma:nearby_cocore_is_a_disk} with these exact Lagrangian fillings shows that $L|_{B^{2n}}$ must be diffeomorphic to a disk, and hence $L$ must be diffeomorphic to $S^n$, since the gluing diffeomorphism in the Weinstein handle decomposition of $T^\ast S^n$ is isotopic to the identity.
        \end{rem}
		\begin{notn}\label{notn:handle_attach}
			Let $L \subset X$ be a nearby Lagrangian cocore. By abuse of notation, let $L \subset X_0$ denote the exact Lagrangian filling of a Legendrian unknot in a Darboux ball in $\partial_\infty X_0$ obtained from \cref{lem:nearby_filling}.
			\begin{enumerate}
				\item Denote by $\widehat X_0$ the Weinstein sector obtained by attaching a critical Weinstein handle $H$ to $X_0$ along $\partial_\infty L$.
				\item Denote by $\widehat L \subset \widehat X_0$ the exact Lagrangian homotopy sphere obtained by attaching the core disk of $H$ to $L$ along $\partial_\infty L$.
				\item Denote by $\widehat C \subset \widehat X_0$ the exact Lagrangian homotopy sphere obtained by attaching the core disk of $H$ to $C_\varepsilon$ along $\partial_\infty C_\varepsilon = \partial_\infty L$.
				\item Denote by $F \subset \widehat X_0$ the Lagrangian cocore disk of $H$.
			\end{enumerate}
		\end{notn}
        \begin{lem}\label{lem:maslov_data_hat}
            If $X$ is equipped with a choice of $R$-Maslov data, there is an induced choice of $R$-Maslov data on $\widehat X_0$.
        \end{lem}
        \begin{proof}
            Denote by 
            \[ \Gr(X)^\# \colon X \longrightarrow B(U/O)^\#\]
            the $R$-Maslov data on $X$ as well as its restriction to $X_0$. Recall from \eqref{eq:two_null_homotopies} that we have a lift $TC_\varepsilon^\# \colon C_\varepsilon \to B(U/O)^\#$. Since $C_\varepsilon$ is a disk, both $\Gr(X)^\#|_{C_\varepsilon}$ and $TC_\varepsilon^\#$ are null-homotopic, and the same is true when taking the pullback of the inclusion $\partial_\infty C_\varepsilon \hookrightarrow C_\varepsilon$.
            
            Consider the commutative diagram 
            \[
                \begin{tikzcd}[column sep=2cm, row sep=2cm]
                    \partial_\infty C_\varepsilon \ar[r,shift left,"\Gr(X_0)^\#"] \ar[r,shift right,swap, "TC_\varepsilon^\#"] \ar[d,hook] & B(U/O)^\# \ar[d]\\
                    \widehat C \smallsetminus \mathring C_\varepsilon \ar[r,"\Gr(\widehat X_0)|_{\widehat C \smallsetminus \mathring C_\varepsilon}"] \ar[ur,"T(\widehat C \smallsetminus \mathring C_\varepsilon)^\#",sloped] & B(U/O)
                \end{tikzcd}
            \]
            Since the two top horizontal arrows are homotopic, there is an extension of $\Gr(X_0)^\#|_{\partial_\infty C_\varepsilon}$ to 
            \[ \Gr(\widehat X_0)^\#|_{\widehat C \smallsetminus \mathring C_\varepsilon} \colon \widehat C \smallsetminus \mathring C_\varepsilon \longrightarrow B(U/O)^\#,\]
            equipped with a homotopy to $T(\widehat C \smallsetminus \mathring C_\varepsilon)^\#$. Since $\widehat X_0 = X_0 \cup_{\partial_\infty C_\varepsilon} H$, and $H\simeq \widehat C \smallsetminus \mathring C_\varepsilon$ (see \cref{dfn:weinstein_handle}), the result follows.
        \end{proof}
        %\sayY{Explain notation below}\sayJ{I added some stuff. I noticed that the definition of $HW(\varOmega Q,L;R)$ at the very end of p.\@ 103 in ADP1 has a typo. It should include $\hocolim_k$ and the action filtration. I will write it down for a future revision.}
        The notion of a flow category equipped with a local system and the geometric realizations of such a category was defined in \cite[Subsection 2.9]{asplund2024nearby}. Let $L,Q \in \Ob(\sW(X;R))$ be two Lagrangian $R$-branes. In \cite[Definition 11.4]{asplund2024nearby} we defined a local system $\sE_{\varOmega Q}$ of $(\varSigma^\infty_+ \varOmega K \wedge R)$-modules on the $R$-oriented flow category $\sC\sW(Q,L)$. The resulting $R$-oriented flow category with local system $(\sC\sW(Q,L),\sE_{\varOmega Q})$ is denoted by $\sC\sW(\varOmega Q,L)$, and its associated CJS realization cf.\@ \cite[Definition 2.39]{asplund2024nearby} is denoted by $HW(\varOmega Q,L;R)$.
        \begin{lem}\label{lma:loop_fiber}
            Suppose that $L \subset \widehat X_0$ is any Lagrangian $R$-brane. There is an equivalence of $R$-modules
            \[
            HW(\varOmega \widehat C,L;R) \simeq HW(F,L;R).
            \]
        \end{lem}
        \begin{proof}
            This is almost identical to the proof of \cite[Proposition 11.10]{asplund2024nearby}. Using the same notation as there, we pick a strictly increasing sequence $\{A_k\}_{k=1}^\infty \subset \IR$ diverging to $\infty$, and get a map
            \[
            |\sN^k,\fm^k,\sE^k| \colon HW^{\leq A_k}(F,L) \longrightarrow HW(\varOmega \widehat C,L),
            \]
            which coincides with the map $\sF^1$ as defined in \cite[Lemma 4.5]{abouzaid2012wrapped}, after passing to appropriate action filtrations   and Morse chain complexes (with coefficients in $k \coloneqq \pi_0R$). This is a quasi-isomorphism, since there is are quasi-equivalences
            \[
            \sW(\widehat X_0;k) \simeq \sW(T^\ast \widehat C;k) \simeq \Perf C_{-\bullet}(\varOmega \widehat C;k),
            \]
            where the first quasi-equivalence holds since the wrapped Fukaya category is invariant up to quasi-equivalence under subcritical handle attachments, cf.\@ \cref{lem:invariance_subcrit} and \cite[Corollary 1.29]{ganatra2022sectorial}. The result follows after passing to homotopy colimits as $k \to \infty$.
        \end{proof}
		\begin{lem}\label{lem:closures_quiv}
			There is a choice of $MO\langle k+1\rangle$-Maslov data $\widehat C^\#$ on $\widehat C$ such that $\widehat L \cong \widehat C^\#$ in $\sW(\widehat X_0;MO\langle k+1\rangle)$.
		\end{lem}
		\begin{proof}
            Let $\sF(\widehat X_0;MO\langle k+1\rangle) \subset \sW(\widehat X_0;MO\langle k+1\rangle)$ denote the full subcategory of compact Lagrangian $R$-branes. We have a commutative diagram
            \[
                \begin{tikzcd}[row sep=scriptsize,column sep=1.5cm]
                    \sF(\widehat X_0;MO\langle k+1\rangle) \dar{\Hw} \rar{\sY_{\varOmega \widehat C}} & \Ho\mod{(\varSigma^\infty_+ \varOmega \widehat C \wedge MO\langle k+1\rangle)} \dar{\Hw} \\
                    \sF(\widehat X_0;H\IZ) \rar{\sY_{\varOmega \widehat C} \wedge_{MO\langle k+1\rangle} H\IZ} & \Ho\mod{(\varSigma^\infty_+ \varOmega \widehat C \wedge H\IZ)}
                \end{tikzcd},
            \]
            where $\Hw$ are the appropriate functors induced by the Hurewicz map $MO \langle k+1\rangle \to H\pi_0MO \langle k+1\rangle\simeq H\IZ$, and where $\sY_{\varOmega \widehat C}$ is the Yoneda functor defined in \cite[Section 11]{asplund2024nearby}. The wrapped Fukaya category of $\widehat X_0$ with integer coefficients is generated by the Lagrangian cocore $F$ by \cite[Theorem 1.13]{ganatra2022sectorial} or \cite[Theorem 1.1]{chantraine2017geometric}, which yields that the Yoneda functor
            \[
            \sY_F \colon \sF(\widehat X_0;\IZ) \longrightarrow \Ho\mod{C_{\bullet}(\varOmega \widehat C;\IZ)}, \quad L \longmapsto HW^{-\bullet}(F,L),
            \]
            is fully faithful; here the $HW(F,F)$-module structure on $HW(F,L)$ is induced by the usual $A_\infty$-module structure obtained from the composition in the integral wrapped Fukaya category. By the construction of $\sY_{\varOmega \widehat C}$ in \cite[Section 11]{asplund2024nearby}, and \cite[Theorem 1.1]{abouzaid2012wrapped}, it follows that the bottom horizontal arrow in the diagram above is fully faithful.
            
            Next, by \cref{lma:loop_fiber} we have 
            \begin{align*}
                HW(\varOmega \widehat C,\widehat C;MO\langle k+1\rangle) &\simeq HW(F,\widehat C;MO\langle k+1\rangle) \simeq \varSigma^\ell MO\langle k+1\rangle \\
                HW(\varOmega \widehat C,\widehat L;MO\langle k+1\rangle) &\simeq HW(F,\widehat L;MO\langle k+1\rangle) \simeq \varSigma^\ell MO\langle k+1\rangle,
            \end{align*}
            for some $\ell \in \IZ$, since both $\widehat C$ and $\widehat L$ intersect $F$ geometrically once, meaning that
            \[
            |\Ob(\sC\sW(F,\widehat C))| = 1 = |\Ob(\sC\sW(F,\widehat L))|.
            \]
            Similar to the proof of \cite[Lemma 12.5]{asplund2024nearby}, choices of $MO\langle k+1\rangle$-Maslov data on $\widehat C$ form a torsor over $[\widehat C,\Pic(MO\langle k+1\rangle)]$. Since $\widehat C$ is a sphere of high enough dimension, we have $[\widehat C,\Pic(MO\langle k+1\rangle)] \simeq [\varOmega \widehat C,BGL_1(MO\langle k+1\rangle)]$. The $(\varSigma^\infty_+ \varOmega \widehat C \wedge MO \langle k+1\rangle)$-module structure on $HW(\varOmega \widehat C,\widehat C;MO\langle k+1\rangle)$ is by adjunction equivalent to a map $(\varOmega \widehat C)_+ \to BGL_1(MO\langle k+1\rangle)$. Therefore we find a choice of $MO\langle k+1\rangle$-Maslov data on $\widehat C$, say $\widehat C^\#$, such that 
            \[
            HW(\varOmega \widehat C,\widehat C^\#;MO\langle k+1\rangle) \simeq HW(\varOmega \widehat C,\widehat L;MO \langle k+1\rangle)
            \]
            as $(\varSigma^\infty_+ \varOmega \widehat C \wedge MO \langle k+1\rangle)$-modules, thus finishing the proof, after applying Whitehead's theorem as in \cite[Lemma C.5]{asplund2024nearby}.
		\end{proof}
        \begin{rem}
            \Cref{lem:closures_quiv} may also be proven by appealing to \cite[Theorem 1.2]{porcelli2025bordism}
        \end{rem}

%% file: Proof.tex
\section{Proof of main result and applications}\label{sec:proof}
    We denote the Lagrangian core of $X$ by $\Core X$, and denote the retraction $X \to \Core X$ by $\pi$. A choice of Weinstein handlebody decomposition of $X$ yields that $\Core X$ admits the structure of an $n$-dimensional CW complex; we denote its $\ell$-skeleton by $(\Core X)_\ell$.

    Let $\boldsymbol{C}$ denote the union of all Lagrangian cocores (and linking disks) of $X$. Note that $\pi(L) \subset \Core X \smallsetminus \pi(\boldsymbol{C}) \simeq (\Core X)_{n-k}$, where $\pi$ is the retract from $X$ to its core, and where $n-k$ is the maximal index of the Weinstein handles of the subcritical part $X_0$. Furthermore, $\pi(\partial_\infty L) \subset \Core X \smallsetminus \pi(\boldsymbol{C})$ is contractible which means that $\pi|_L$ restricts to a well-defined continuous map $L/\partial_\infty L \to \Core X \smallsetminus \pi(\boldsymbol{C})$ that we by abuse of notation keep denoting by $\pi|_L$.

    \begin{defn}\label{dfn:relative_brane}
        We say that a nearby Lagrangian cocore $L \subset X$ admits \emph{relative $R$-Maslov data} if there is a null-homotopy of the map $TL^\#-\Gr(X)^\#|_L$, defined in \eqref{eq:lift_maslov_data_L}, that is constant when restricted to $\partial_\infty L$.
    \end{defn}
    Recall that $O\langle k+1\rangle \to O$ denotes the $k$-connected cover of the orthogonal group $O$, and its associated Thom spectrum is denoted by $MO\langle k+1\rangle$.
    \begin{asmpt}\label{asmpt:main_asmpt}
        The Weinstein sector $X$ admits $MO\langle k+1\rangle$-Maslov data, and the nearby Lagrangian cocore $L \subset X$ admits relative $MO\langle k+1\rangle$-Maslov data.
    \end{asmpt}

    \begin{thm}\label{thm:main}
        Let $k \geq 1$ and let $n \geq 2k+2$. Suppose that $X^{2n}$ is a Weinstein sector with a chosen Weinstein handlebody decomposition such that the maximal index of a subcritical Weinstein handle is $\leq n-k$. Let $L \subset X$ be a nearby Lagrangian cocore such that \cref{asmpt:main_asmpt} is satisfied. The following composition is null-homotopic
        \[
        L/\partial_\infty L \overset{\pi|_L}{\longrightarrow} \Core X \smallsetminus \pi(\boldsymbol{C})  \longrightarrow (\Core X)_{n-k}/(\Core X)_{n-k-1}.
        \]
    \end{thm}
    \begin{proof}
        Recall the construction of $\widehat X_0$ and the exact Lagrangian homotopy spheres $\widehat L, \widehat C \subset \widehat X_0$ from \cref{notn:handle_attach}. First by our assumption \cref{asmpt:main_asmpt}, $X$ admits $MO\langle k+1\rangle$-Maslov data, and \cref{lem:maslov_data_hat} implies that so does $\widehat X_0$. By construction, the $MO\langle k+1\rangle$-Maslov data on $X$ is so that $TC_\varepsilon^\# - \Gr(X)^\#|_{C_\varepsilon}$ is the constant map, which is to say that $\widehat C \subset \widehat X_0$ admits $MO \langle k+1\rangle$-Maslov data. Furthermore, the assumption that $L$ admits a choice of relative $MO\langle k+1\rangle$-Maslov data implies that $\widehat L \subset \widehat X_0$ admits a choice of $MO\langle k+1\rangle$-Maslov data.
        
        Next, by \cref{lem:closures_quiv}, we have that $\widehat L \cong \widehat C$ in $\sW(\widehat X_0;MO\langle k+1\rangle)$ for some choice of $MO\langle k+1\rangle$-Maslov data on $\widehat C$. Since $\widehat L$ is a homotopy $n$-spheres, we choose a tangential $O\langle k+1\rangle$-structure which yields a choice of $MO\langle k+1\rangle$-fundamental class $[\widehat L] \in H_n(X_0;MO\langle k+1\rangle)$. By \cref{thm:oc_fund_classes}, there is a choice of $MO\langle k+1\rangle$-fundamental class $[\widehat C]$ for $\widehat C$, such that $[\widehat C] = [\widehat L] \in H_n(X_0;MO\langle k+1\rangle)$. Let $\widehat{\pi} \colon \widehat{X}_0 \to \Core X_0$ denote the fold map as defined in \cref{sec:fold_map} that extends the retraction $\pi \colon X_0 \to \Core X_0$. It moreover has the property that $\widehat{\pi}|_{\widehat{L}}$ is null-homotopic if and only if $\pi|_L$ is null-homotopic, see \cref{lem:Folding}. Since $\pi(C) = \{\mathrm{pt}\} \subset \Core X$, and it follows that after flowing with the backwards Liouville flow as in \cref{lem:nearby_filling} so that $C \subset X_0$, we have that $\pi(C) \subset \Core X_0$ is contractible. Hence we have $\widehat{\pi}_\ast[\widehat{C}] = 0$ in $\widetilde H_n(\Core X_0;MO\langle k+1\rangle)$ and thus $\widehat{\pi}_\ast[\widehat{L}] = 0$ in $\widetilde H_n(\Core X_0;MO\langle k+1\rangle)$. The result now follows from \cref{lem:DetectingNullHomotopy_gen}.
    \end{proof}
    \begin{rem}
        The positive integer $k$ in \cref{thm:main} is the positive integer such that the handle of maximal dimension of $X_0$ is $n-k$. The coefficients of the wrapped Donaldson--Fukaya category we use depend on $k$, and below we list the first few cases. 
        \begin{description}
            \item[$k\in \{1,2\}$] We require $n\geq 4$ or $n\geq 6$ and use coefficients in $MSpin$.
            \item[$k\in \{3,4,5,6\}$] We require $n\geq 2k+2$ and use coefficients in $MString$.
            \item[$k=7$] We require $n\geq 14$ and use coefficients in $MFiveBrane$.
        \end{description}
        For $p > q$, there is a ring map $MO\langle p \rangle \to MO \langle q\rangle$, which means that if $X$ admits $MO\langle p \rangle$-Maslov data, it also admits $MO\langle q \rangle$-Maslov data.
    \end{rem}
    \subsection{Smooth unknotting of nearby Lagrangian cocores}\label{sec:smooth_unknotting}
    \begin{defn}
        We say that a nearby Lagrangian cocore $L \subset X$ is \emph{smoothly unknotted} if it is smoothly isotopic (rel boundary) to a Lagrangian cocore, in the complement of all Lagrangian cocores.
    \end{defn}
    \begin{cor}\label{cor:cp2}
        Any nearby Lagrangian cocore $L \subset T^\ast \CP^2$ is smoothly unknotted.
    \end{cor}
    \begin{proof}
        First, $T^\ast \CP^2$ admits a polarization given by the tangent spaces of the cotangent fibers, and hence $\sW(T^\ast \CP^2;MSpin)$ is well-defined, see \cite[Theorem 1.2]{asplund2024nearby}. We consider the map
        \[
        \pi|_L \colon L/\partial_\infty L \longrightarrow \CP^2 \smallsetminus \{\mathrm{pt}\} = S^2.
        \]
        Like in the proof of \cref{thm:main} we obtain that $\widehat{\pi}_\ast[L] = 0$ in $\widetilde{H}_4(S^2;MSpin) \cong \pi_2^{\mathrm{st}}$. By \cref{lem:DetectingNullHomotopy_gen} together with the fact that we have early stabilization $\pi^{\mathrm{st}}_2 \cong \pi_4(S^2)$, it follows that $\pi|_L$ is null-homotopic. If we denote by $F \subset T^\ast \CP^2$ a cotangent fiber that is disjoint from $F$, we have that $L \to T^\ast \CP^2 \smallsetminus F$ and $C_\varepsilon \to T^\ast \CP^2 \smallsetminus F$ are homotopic rel boundary. The $h$-principle \cite[Théorème d'existence]{haefliger1961plongements} finally shows that they are in fact smoothly isotopic rel boundary.
    \end{proof}
    \begin{thm}\label{thm:smooth_iso}
		Let $k\geq 1$, and $n \geq 2k+2$. Let $N$ be a $k$-dimensional $K(\pi,1)$ manifold. Let $M$ be a smooth $(n-k)$-dimensional manifold satisfying
        \begin{enumerate}
            \item $\pi_2(M) = 0$, and 
            \item $\pi_n(M) = 0$ or $\widetilde M \simeq \bigvee_i S^{n-k_i}$ where $k_i \leq k$.
        \end{enumerate}
        Any nearby Lagrangian cocore $L \subset T^\ast(M^{n-k} \times N^k)$ is smoothly unknotted.
	\end{thm}
	\begin{proof}
        Nearby Lagrangian cocores in $T^\ast(M^{n-k} \times N^k)$ admit a lift to the cotangent bundle of the universal cover $T^\ast(\widetilde M \times \IR^k)$. Let $p \colon \widetilde M \times \widetilde N \to M \times N$ denote the covering map. Then we have $(M \times \IR^k) \smallsetminus \{p^{-1}(\mathrm{pt})\} \simeq \widetilde M \vee \bigvee_{\varPi} S^{n-1}$, where $\varPi \coloneqq \pi_1(M) \times \pi_1(N)$. We therefore obtain the map
		\[
			\pi|_L \colon L/\partial_\infty L \longrightarrow \widetilde M \vee \bigvee_{\varPi} S^{n-1},
		\]
        which has the property that
		\[
			L/\partial_\infty L \overset{\pi|_L}{\longrightarrow} \widetilde M \vee \bigvee_{\varPi} S^{n-1} \overset{q_i}{\longrightarrow} S^{n-1},
		\]
        is null-homotopic by \cref{thm:main}, for each $i$, where $q_i \colon \bigvee_{\varPi} S^{n-1} \to S^{n-1}$ is the $i$-th projection. We also have $L/\partial_\infty L \xrightarrow{\pi|_L} \widetilde M \vee \bigvee_{\varPi} S^{n-1} \to \widetilde M$ is either null-homotopic by the assumption $\pi_n(M) = 0$, or by the proof of \cref{thm:main} if we assume $\widetilde M \simeq \bigvee_i S^{n-k_i}$. The homotopy class of the map $\pi|_L$ defines an element of $\pi_n\left(\widetilde M \vee \bigvee_{\varPi} S^{n-1}\right)$, and we will now show that the induced map 
		\begin{equation}\label{eq:comp_maps_pi_n}
		    \pi_n\left(\widetilde M \vee \bigvee_{\varPi} S^{n-1}\right) \longrightarrow \pi_n \left(\widetilde M \times \bigvee_{\varPi} S^{n-1}\right) \cong \pi_n(\widetilde M) \oplus \bigoplus_{\varPi} \pi_n(S^{n-1}),
		\end{equation}
		is in fact an isomorphism. The first map is induced by the component maps $\widetilde M \vee \bigvee_{\varPi} S^{n-1} \to \widetilde M$ and $\widetilde M \vee \bigvee_{\varPi} S^{n-1} \to S^{n-1}$. From its cell decomposition it is clear that $\bigvee_{\varPi} S^{n-1}$ is $(n-2)$-connected. Now, since $\widetilde M$ is $2$-connected by assumption, we have that $\varSigma\left(\varOmega \widetilde M \wedge \varOmega \bigvee_{\varPi} S^{n-1}\right)$ is $n$-connected by \cref{lem:SomeElementaryATFacts}(ii). Therefore it follows that \eqref{eq:comp_maps_pi_n} is an isomorphism, and thus $\pi|_L$ is null-homotopic. We finish the proof by applying the $h$-principle \cite[Théorème d'existence]{haefliger1961plongements} as in \cref{cor:cp2}.
	\end{proof}
    \begin{por}
        Any nearby Lagrangian cocore $L \subset T^\ast(S^5 \times S^5 \times \IR)$ or $L \subset T^\ast(S^5 \times S^5 \times S^1)$ is smoothly unknotted.
    \end{por}
    \begin{proof}
        We consider the first case $L \subset T^\ast(S^5 \times S^5 \times \IR)$. We have
        \[
            (S^5 \times S^5 \times \IR) \smallsetminus \{\mathrm{pt}\} \simeq (S^5 \times S^5) \vee S^{10},
        \]
        and by the proof of \cref{thm:smooth_iso}, the map
        \[
        \pi_{11}((S^5 \times S^5) \vee S^{10}) \longrightarrow \pi_{11}(S^5 \times S^5) \oplus \pi_{11}(S^{10})
        \]
        is injective. Noting that by the early stabilization phenomenon, we have $\pi_{11}S^5 \cong \pi_6^{\mathrm{st}}$. The rest of the proof is as in \cref{thm:smooth_iso}.

        For the second case $L \subset T^\ast(S^5 \times S^5 \times S^1)$ we pass to universal covers as in the proof of \cref{thm:smooth_iso}; the rest is completely analogous to the above.
    \end{proof}
    \begin{rem}
		Examples of smooth manifolds $M \times N$ such that any nearby Lagrangian cocore in $T^\ast(M \times N)$ is smoothly unknotted include the following:
		\begin{itemize}
            \item $S^{n-k} \times \IR^k$ and $S^{n-k} \times T^k$.
            \item $(S^{n-k-1} \times S^1)^{\# i} \times \IR^k$ and $(S^{n-k-1} \times S^1)^{\# i} \times T^k$.
            \item $\RP^{n-k} \times \IR^k$ and $\RP^{n-k} \times T^k$.
            \item $(\RP^{n-k-1} \times S^1)^{\# i} \times \IR^k$ and $(\RP^{n-k-1} \times S^1)^{\# i} \times T^k$.
            \item $L(p;q_1,\ldots,q_n) \times \IR^k$ and $L(p;q_1,\ldots,q_n) \times T^k$, where $L(p;q_1,\ldots,q_n)$ denotes any higher dimensional lens space.
            \item $\OP^2 \times \IR^4$ and $\OP^2 \times T^4$ (since $\pi_{20}(\OP^2) \cong 0$ by \cite[Theorem 7.2]{mimura1967homotopy}).
		\end{itemize}

        Notably, we do not know whether a nearby Lagrangian cocore inside $T^\ast(S^{n-k} \times S^k)$ for $k\geq 2$ is smoothly unknotted in the complement of a cotangent fiber; however, we do have smooth unknotting not necessarily in the complement of a fiber in the special cases of $S^2 \times S^2$ and $S^6 \times S^6$, see \cref{por:prod_spheres}.
	\end{rem}
	\begin{thm}\label{thm:conn_sum}
		Let $k\geq 1$, and let $n\geq 2k+2$. Let $M$ and $N$ be smooth $n$-manifolds and denote by $\mathring M$ and $\mathring N$ the complement of a point in $M$ and $N$, respectively. Suppose that $N$ satisfies either $\mathring N \simeq \bigvee_i S^{n-k_i}$ for $k_i \leq k$ or $\mathring N \simeq \ast$. Assume that $M$ satisfies either:
        \begin{enumerate}
            \item $\pi_n(\mathring{\widetilde M}) = 0$ and $\mathring{\widetilde M}$ is $(k+1)$-connected if $\mathring N \simeq \bigvee_i S^{n-k_i}$, or
            \item $\mathring{\widetilde M} \simeq \bigvee_i S^{n-\ell_i}$ for some $\ell_i \leq k$.
        \end{enumerate}
        Then any nearby Lagrangian cocore $L \subset T^\ast(M \# N)$ is smoothly unknotted.
	\end{thm}
    \begin{proof}
        Similar to the proof of \cref{thm:smooth_iso}, nearby Lagrangian cocores in $T^\ast(M \# N)$ lift to nearby Lagrangian cocores in $T^\ast(\widetilde{M \# N})$. If we denote the universal covering map by $p$, we have
        \[
        \widetilde{M \# N} \smallsetminus \{p^{-1}(\mathrm{pt})\} \simeq \mathring{\widetilde M} \vee \bigvee_{\pi_1(M)} \bigvee_i S^{n-k_i}.
        \]
        By the Seifert--van Kampen theorem it follows by assumption that $N$ is simply connected. Consequently we have $\mathring N \simeq \bigvee_i S^{n-k_i}$. We see that $\bigvee_{\pi_1(M)}\bigvee_i S^{n-k_i}$ is $(n-k-1)$-connected, and since $\mathring{\widetilde M}$ is at least $(k+1)$-connected, it follows that $\varSigma(\varOmega \mathring{\widetilde M} \wedge \varOmega \bigvee_i S^{n-k_i})$ is $n$-connected, and therefore the induced map
        \[
        \pi_n\left(\mathring{\widetilde M} \vee \bigvee_{\pi_1(M)}\bigvee_i S^{n-k_i}\right) \longrightarrow \pi_n\left(\mathring{\widetilde M}\times \prod_{\pi_1(M)}\prod_i S^{n-k_i}\right) \cong \pi_n(\mathring{\widetilde M}) \oplus \bigoplus_{\pi_1(M)}\bigoplus_i\pi_n(S^{n-k_i})
        \]
        is injective. Now, the composition
        \[
        L/\partial_\infty L \overset{\pi|_L}{\longrightarrow} \mathring{\widetilde M} \times \prod_{\pi_1(M)}\prod_i S^{n-k_i} \longrightarrow \mathring{\widetilde M},
        \]
        is either null-homotopic by the assumption $\pi_n(\mathring{\widetilde M}) = 0$, or by the proof of \cref{thm:main} if we assume $\mathring{\widetilde M} \simeq \bigvee_i S^{n-\ell_i}$ for $\ell_i \leq k$. The compositions
        \[
        L/\partial_\infty L \overset{\pi|_L}{\longrightarrow} \mathring{\widetilde M} \times \prod_{\pi_1(M)}\prod_i S^{n-k_i} \longrightarrow S^{n-k}
        \]
        are null-homotopic by \cref{thm:main}. The rest of the proof follows the proof of \cref{thm:smooth_iso}.
	\end{proof}
    \begin{cor}\label{cor:conn_sum_aspherical}
        Let $k\geq 1$, and let $n \geq 2k+2$. Let $M$ be a $K(\pi,1)$ $n$-manifold, and let $N$ be a smooth $n$-manifold such that $\mathring N \simeq \bigvee_i S^{n-k_i}$ for $k_i \leq k$ or $\mathring N \simeq \ast$. Then any nearby Lagrangian cocore $L \subset T^\ast(M \# N)$ is smoothly unknotted.
    \end{cor}
    \begin{proof}
        The universal cover of $M$ is $\IR^n$, and the complement of a point is homotopy equivalent to $S^{n-1}$.
    \end{proof}
    \begin{rem}
        Below let $k\geq 1$ and $n\geq 2k+2$. By \cref{thm:conn_sum,cor:conn_sum_aspherical}, if we let $\varSigma^n$ denote any homotopy $n$-sphere, nearby Lagrangian cocores in cotangent bundles of any of the following smooth manifolds are smoothly unknotted:
        \begin{itemize}
            \item $(S^{n-k} \times \IR^k)\# \varSigma^n$
            \item $\IR^n \# \varSigma^{n}$
            \item $\CP^2 \# \varSigma^{4}$
            \item $(S^5 \times S^5 \times \IR) \# \varSigma^{11}$
            \item $T^n \# (S^{n-k} \times \IR^k)$
            \item $(S^{n-k} \times \IR^k) \# (S^{n-k} \times \IR^k)$
        \end{itemize}
        Since $(S^{n-k} \times \IR^k) \# \varSigma^n$ and $(S^{n-k} \times \IR^k) \# (S^{n-k} \times \IR^k)$ are $(n-k-1)$-connected, and we assume $n\geq 2k+2$, repeated applications of \cref{thm:conn_sum} shows that nearby Lagrangian cocores in $T^\ast((S^{n-k} \times \IR^k)^{\# i} \# (\varSigma^n)^{\# j})$ are smoothly unknotted for any $i,j \geq 0$.
    \end{rem}
    \begin{thm}\label{thm:plumbing}
        Let $M$ and $N$ be smooth $n$-manifolds such that 
        \[
        \mathring M \simeq \bigvee_i S^{n-k_i}, \quad \mathring N \simeq \bigvee_i S^{n-\ell_i}, \quad k_i,\ell_i \geq 1.
        \]
        Let $K \coloneqq \max_i \{k_i,\ell_i\}$, and assume $n \geq 2K+2$.
        \begin{enumerate}
            \item Any nearby Lagrangian cocore $L \subset T^\ast M \# T^\ast N$ in the plumbing of the two cotangent bundles along a point, is smoothly unknotted.
            \item Any nearby Lagrangian cocore $L \subset T^\ast M \#_{\mathrm{cyc}} T^\ast N$ in the plumbing of the two cotangent bundles along two points, is smoothly unknotted.
        \end{enumerate}
    \end{thm}
    \begin{proof}
        \begin{enumerate}
            \item With an appropriate choice of Weinstein handle decomposition of $X = T^\ast M \# T^\ast N$, there are two Lagrangian cocores corresponding to cotangent fibers in each of the two cotangent bundles. Therefore the core of the subcritical part $X_0$ is
            \[
            \Core X_0 \simeq \mathring M \vee \mathring N \simeq \bigvee_i S^{n-k_i} \vee \bigvee_i S^{n-\ell_i}.
            \]
            By our assumption, it follows that both $\bigvee_i S^{n-k_i}$ and $\bigvee_i S^{n-\ell_i}$ are $(n-K-1)$-connected, so the induced map
            \[
            \pi_n\left(\bigvee_i S^{n-k_i} \vee \bigvee_i S^{n-\ell_i}\right) \longrightarrow \pi_n\left(\prod_i S^{n-k_i} \times \prod_i S^{n-\ell_i}\right),
            \]
            is injective. The rest of the proof is similar to the proof of \cref{thm:conn_sum}.
            \item In this case with $X = T^\ast M \#_{\mathrm{cyc}} T^\ast N$, we use a similar idea to that of \cref{thm:smooth_iso} and pass to the universal cover. By assumption, $\mathring M \vee \mathring N$ is simply connected, and we see
            \[
            \Core (\widetilde X)_0 \simeq \widetilde{\Core X_0} \simeq \widetilde{\mathring M \vee \mathring N \vee S^1} \simeq \bigvee_{\IZ} (\mathring M \vee \mathring N).
            \]
            Again, since $\mathring M \vee \mathring N$ is $(n-K-1)$-connected, it follows by induction that
            \[
            \pi_n \left(\bigvee_{\IZ}(\mathring M \vee \mathring N)\right) \cong \bigoplus_{\IZ} \pi_n(\mathring M \vee \mathring N).
            \]
            The rest of the proof proceeds as in item (i).
        \end{enumerate}
    \end{proof}
    \begin{rem}
        For any $k\geq 1$, we have that $M \coloneqq (S^{n-k} \times \IR^k)^{\# i} \# (\varSigma^n)^{\# j}$ is $(n-k-1)$-connected, where $i,j \geq 0$, and $\varSigma^n$ is a homotopy $n$-sphere. Under the assumption, repeated applications of \cref{thm:plumbing} shows that nearby Lagrangian cocores in any plumbing (along points) of copies of $T^\ast M$ are smoothly unknotted.
    \end{rem}
    The proof of \cref{thm:smooth_iso} also shows that we sometimes obtain smooth unknotting not necessarily in the complement of the missed Lagrangian cocores.
    \begin{por}\label{por:prod_spheres}
        Any nearby Lagrangian cocore in $T^\ast(S^2 \times S^2)$ or $T^\ast(S^6 \times S^6)$ is smoothly isotopic, relative to its boundary, to the cotangent fiber (not necessarily in the complement of the missed cotangent fiber).
    \end{por}
    \begin{proof}
        In the first case we consider the map $\pi|_L \colon L/\partial_\infty L \to S^2 \times S^2$, where we do not remove the missed cotangent fiber. Repeating the proof of \cref{thm:main} for the composition of $\pi|_L$ with the two projection maps to each $S^2$-factor, shows that $\pi|_L$ is null-homotopic by the fact that we have early stabilization $\pi_4(S^2) \cong \pi_2^{\mathrm{st}}$. The case with $S^6 \times S^6$ is similar. In that case we also have early stabilization $\pi_{12}(S^6) \cong \pi_6^{\mathrm{st}}$. Repeating the last part of the proof of \cref{thm:smooth_iso}, finishes the proof.
    \end{proof}
    \begin{rem}\label{rem:unkotted_not_in_complement}
        \Cref{por:prod_spheres} also holds true for nearby Lagrangian cocores inside cotangent bundles of $S^2 \times \RP^2$, $\RP^2 \times \RP^2$, $S^6 \times \RP^6$ and $\RP^6 \times \RP^6$ by passing to the universal covers as in the proof of \cref{thm:smooth_iso}.
    \end{rem}
    
    \subsection{Exact Lagrangian fillings}
        By \cref{lem:nearby_filling} a nearby Lagrangian cocore in $X$ corresponds to an exact Lagrangian filling of a Legendrian unknot in a Darboux ball in the contact boundary of $X_0$. The standard contact form on $\IR^{2n-1}$ is given by $\alpha = dz - \sum_{i=1}^{n-1}y_idx_i$. An exact Lagrangian cobordism from a Legendrian submanifold $\varLambda_- \subset \IR^{2n-1}$ to a Legendrian submanifold $\varLambda_+ \subset \IR^{2n-1}$ is an exact Lagrangian submanifold $L \subset (\IR_t \times \IR^{2n-1},d(e^t\alpha))$ such that there exists some $T > 0$ such that
        \begin{align*}
            L \cap ((-\infty,-T) \times \IR^{2n-1}) &= (-\infty,-T) \times \varLambda_- \\
            L \cap ((T,\infty) \times \IR^{2n-1}) &= (T,\infty) \times \varLambda_+.
        \end{align*}
        An exact Lagrangian concordance is an exact Lagrangian cobordism that is diffeomorphic to $S^{n-1} \times \IR$. If such an exact Lagrangian concordance exists we write $\varLambda_- \prec \varLambda_+$. We let $\varLambda_0 \subset \IR^{2n-1}$ be the standard Legendrian unknot. 
        \begin{defn}\label{defn:slice}
            A Legendrian knot $\varLambda \subset \IR^{2n-1}$ is \emph{Legendrian slice} (inside $\IR^{2n-1}$) if $\varLambda_0 \preceq \varLambda$. 
        \end{defn}
        Recall that \cref{defn:slice} is equivalent to the usual definition of Legendrian sliceness that $\varLambda \subset \IR^{2n-1}$ bounds an exact Lagrangian disk inside $\IR^{2n}$, cf.\@ \cite[Remark 4.5]{chantraine2010lagrangian} or \cite[Section 4]{cornwell2016obstructions}.
    
        Suppose that $\varLambda \subset \partial_\infty X_0$ is a Legendrian knot contained in a Darboux ball that is Legendrian slice inside it. As in \cref{thm:main}, the retraction $\pi\colon X \to \Core X$ restricts to a well-defined continuous map $\pi|_L \colon L/\varLambda \to \Core X_0$, since $\varLambda$ is contained in a Darboux ball inside $\partial_\infty X_0$.
        
        The following result is a generalization of \cref{thm:main}.
        \begin{thm}\label{thm:main_conc}
            Let $k\geq 1$ and $n\geq 2k+2$. Suppose that $X^{2n}_0$ is a subcritical Weinstein sector satisfying \cref{asmpt:main_asmpt}. Suppose that $\varLambda \subset \partial_\infty X_0$ is a Legendrian knot contained in a Darboux ball that is Legendrian slice inside it. If $L \subset X_0$ is an exact Lagrangian filling of $\varLambda$, the following composition is null-homotopic
            \[
            L/\varLambda \overset{\pi|_L}{\longrightarrow} \Core X_0  \longrightarrow \Core X_0/(\Core X_0)_{n-k-1}.
            \]
        \end{thm}
        \begin{proof} 
            By \cref{lem:nearby_filling}, the Legendrian unknot $\varLambda_0 \subset \partial_\infty X_0$ has an exact Lagrangian disk filling $C_0 \subset X_0$. If the Lagrangian concordance between $\varLambda$ and $\varLambda_0$ inside $\IR \times \IR^{2n-1}$ is denoted by $C_\varLambda$, we construct an exact Lagrangian disk filling $C \subset X_0$ of $\varLambda$ by gluing $C_0$ and $C_\varLambda$ together along their common Legendrian boundary $\varLambda_0$. By the first part of the proof of \cref{lma:nearby_cocore_is_a_disk}, it therefore follows that $L \subset X_0$ is an integer homology disk.
            Consider the Weinstein sector $\widehat X_0$, and the two Lagrangians $\widehat L, \widehat C \subset \widehat X_0$ as described in \cref{notn:handle_attach}. By construction $\widehat C$ is an $n$-dimensional sphere, and $\widehat L$ is an integer homology sphere by the Mayer--Vietoris exact sequence.

            Next, we show that $\widehat L$ is simply connected. Consider the enlargement of the wrapped Fukaya category (with $\IF_2$-coefficients) $\sS(\widehat X_0)$ whose objects are exact conical Lagrangians equipped with local systems of $\IF_2$-chain complexes of arbitrary dimension, as defined in \cite{abouzaid2012nearby}. By \cite[Theorem 1.4]{chantraine2017geometric} or \cite[Theorem 1.2, Example 1.3]{ganatra2020covariantly}, the critical Lagrangian cocore $F$ inside $\widehat X_0$ resolves the diagonal in the sense that it satisfies Abouzaid's split-generation criterion. Furthermore $HW^{-\bullet}(F,F;\IF_2) \cong H_{\bullet}(\widehat C;\IF_2)$ is supported in non-negative degrees. With these two observations, the proofs of \cite[Lemmas 3.1 and 3.2]{abouzaid2012nearby} go through, to show that $(\widehat L,E)$ is quasi-isomorphic in $\sS(\widehat X_0)$ to $(\widehat C,E_{\widehat C}')$, where $E_{\widehat C}'$ is a trivial local system. Moreover, it also follows from \cite[Lemma 2.13]{abouzaid2012nearby} that trivial local systems on $\widehat C$ are isomorphic in $\sS(\widehat X_0)$ to trivial local systems on $\widehat L$ of the same rank. Finally, \cite[Lemma B.1]{abouzaid2012nearby} shows that a quasi-isomorphism in $\sS(\widehat X_0)$ implies a quasi-isomorphism in the category of local systems on $\widehat L$ whose morphisms are $H^\bullet(\widehat L;\Hom(E,E'))$, which in degree $0$ is generated by global maps of local systems, so that an isomorphism in this category is an isomorphism of local systems in the usual sense. This shows that any local system on $\widehat L$ is trivial, and consequently that $\widehat L$ is simply connected, and hence a homotopy disk.
    
            The rest of the proof is the same as the proof of \cref{thm:main}. We note that the fact that the exact Lagrangian concordance $C_{\varLambda}$ is contained in the symplectization of the Darboux ball in the contact boundary guarantees that $\widehat{\pi}(\widehat{C})$ is contained in a small ball in $\Core X_0$, and hence that the induced map $\pi|_C \colon C/\varLambda \to \Core X_0$, is null-homotopic.
        \end{proof}
        \begin{rem}\label{rem:slice_remark}
            \begin{enumerate}
                \item In case the Chekanov--Eliashberg dg-algebra of $\varLambda$ is supported in degrees $\leq -1$, \cite[Theorem 70]{ekholm2017duality} (cf.\@ \cite[Theorem 1.11]{chantraine2020floer}) offers an alternative proof that $\widehat L$ is simply connected in the proof of \cref{thm:main_conc}.
                \item Combining \Cref{thm:main_conc} with results in \cref{sec:smooth_unknotting} yields a smooth isotopy (rel boundary) classification of exact Lagrangian fillings of Legendrian slice knots in boundaries of some subcritical Weinstein sectors.
            \end{enumerate}
        \end{rem}

%% file: Fold.tex
\section{Fold map}\label{sec:fold_map}
In this section we give a careful construction of the continuous extension of the deformation retract $\pi \colon X_0 \to \Core X_0$ to a continuous map $\widehat \pi \colon \widehat X_0 \to \Core X_0$.
\begin{notn}
    Let $X_0$ be a subcritical Weinstein manifold and 
    \[ i\colon L \hooklongrightarrow X_0 \qquad \text{and} \qquad j\colon K \hooklongrightarrow X_0\]
    be exact conical Lagrangians such that $\varLambda \coloneqq \partial_\infty L = \partial_\infty K$ and let
    \[ \widehat{i} \colon \widehat{L} \longrightarrow \widehat X_0 \qquad \text{and} \qquad \widehat{j} \colon \widehat{K} \longrightarrow \widehat X_0\]
    denote the associated inclusions, respectively, where $\widehat L$, $\widehat K$ and $\widehat X_0$ are obtained from $L$ and $K$ by attachment along the boundary of the core disk of a critical Weinstein handle, see \cref{notn:handle_attach}.
\end{notn}
Since $X_0$ does not have any critical Weinstein (half)-handle, its core is naturally a CW complex of dimension $\leq n-1$, and we consider the quotient map
\[
q \colon \Core X_0 \longrightarrow \Core X_0/(\Core X_0)_{n-2}.
\]
We now extend the composition
\[
    \begin{tikzcd}[row sep=scriptsize, column sep=scriptsize]
        X_0 \rar{\pi} & \Core X_0 \rar{q} & \Core X_0/(\Core X_0)_{n-2}
    \end{tikzcd}
\]
to a map $\widehat X_0 \to \Core X_0/(\Core X_0)_{n-2}$ by ``folding $\widehat X_0 \smallsetminus X_0$ down to a point.''
\begin{notn}
    Let $X^{\mathrm{in}}$ and $X^{\mathrm{in}}_0$ denote a fixed choice of defining Weinstein domain for $X$ and $X_0$, respectively, such that $L$ and $K$ are conical outside of $X^{\mathrm{in}}_0$. Let
    \begin{align*}
        L^{\mathrm{in}} &\coloneqq L \cap X^{\mathrm{in}}_0 \\
        K^{\mathrm{in}} &\coloneqq K \cap X^{\mathrm{in}}_0.
    \end{align*}
\end{notn}
Let $Y$ be a topological space. Suppose that there is a map $f\colon X_0^{\mathrm{in}} \to Y$ and suppose that there exists a null-homotopy of the composition $f \circ i|_{\varLambda}$ given by
\begin{equation}\label{eq:null_htpy}
    H \colon \varLambda \times [0,1] \longrightarrow Y,
\end{equation}
where $H(\varLambda \times \{0\}) = f \circ i|_{\varLambda}$ and $H(\varLambda \times \{1\}) = y$, where $y\in Y$ is a fixed point. This null-homotopy induces a map
\[ f_K \colon \frac{K^{\mathrm{in}} \cup_{\varLambda \times \left\{0\right\}} (\varLambda \times [0,1])}{\varLambda \times \left\{1\right\}} \longrightarrow Y \]
given by
\begin{equation}\label{eq:Folding}
    f_K(p) = \begin{cases} f \circ j(p), & p \in K^{\mathrm{in}} \\ H(p), & p \in \varLambda \times [0,1]\end{cases}.
\end{equation}
The homotopy class of $f_K$ only depends on the homotopy classes of the maps $j$ and $H$.
\begin{lem}\label{lem:Folding}
There exists a continuous map $\widehat{f} \colon \widehat X_0 \to Y$ that satisfies:
\begin{enumerate}
    \item $\widehat{f}|_{X^{\mathrm{in}}_0} = f$ and
    \item $\widehat{f} \circ \widehat{j}$ is null-homotopic if and only if $f_K$ is null-homotopic.
\end{enumerate}
\end{lem}

\begin{proof}
We construct $\widehat{f}$ as follows. Recall from \cref{notn:handle_attach} that $\widehat K$ is obtained by attachment along the boundary of the core disk of a critical Weinstein handle $H$, denoted by $D$. Extend the null-homotopy $H$ in \eqref{eq:null_htpy} to a smooth map $H \colon D \to Y$ such that $H(x) = y$ for all $x$ outside of a collar neighborhood of the boundary that is diffeomorphic to $\varLambda \times [0,1]$. Consider the projection $p\colon H \to D$ from the Weinstein handle to the core disk, and define
\[ \widehat{f} \colon \widehat X_0 \longrightarrow Y,\quad x \longmapsto \begin{cases} H(p(x)), & x \in H \\ f(x), & x \in X^{\mathrm{in}} \\ \end{cases}.\]
Item (i) now follows by definition of $\widehat f$.  For item (ii), notice that we have a commutative diagram
\[ \begin{tikzcd}[row sep=scriptsize] \widehat{K} \arrow[rd,"\widehat{f} \circ \widehat{j}",bend left=15] \arrow[d] & \\ \frac{K^{\mathrm{in}} \cup_{\varLambda \times \{0\}} (\varLambda \times [0,1])}{\varLambda \times \{1\}} \arrow[r, swap,"f_K"]& Y, \\ \end{tikzcd} \]
where the vertical map is defined by writing
\[
\widehat{K} = K^{\mathrm{in}} \cup_{\varLambda \times \{0\}} (\varLambda \times [0,1]) \cup_{\varLambda\times \{1\}} D
\]
and collapsing $D$ to a single point in $\varLambda \times [0,1]$. From this commutativity, it follows that $\widehat{f} \circ \widehat{j}$ is null-homotopic whenever $f_K$ is null-homotopic.  For the converse direction, suppose that $G\colon \widehat{K} \times [0,1] \to Y$ is a null-homotopy of $\widehat{f} \circ \widehat{j}$.  Without loss of generality, assume that $G(x,t) = (\widehat{f} \circ \widehat{j})(x)$ for some $x \in \widehat{K} \smallsetminus K^{\mathrm{in}}$. Pick an identification $\Nbhd_{\widehat K}(\varLambda) \cong \varLambda \times [0,1] \xrightarrow{\pi_{[0,1]}} [0,1]$ using the Liouville flow on $\widehat X$ and let $\ell\colon \widehat K \to [0,1]$ be a constant extension of $\pi_{[0,1]}$ to all of $\widehat K$ defined by $\ell(K^{\mathrm{in}}) = 0$ and $\ell(D) = 1$. Define a homotopy
\begin{align*}
    F\colon \widehat{K} \times [0,1] &\longrightarrow Y \\
    (x,t) &\longmapsto G(x,t-\ell(x)).
\end{align*}
Notice that $F_t$ is constant on $\widehat{K} \smallsetminus K^{\mathrm{in}}$.  Hence it descends to a null-homotopy of $f_K$, as desired.
\end{proof}

\begin{rem}
The homotopy class of $\widehat{f}$ only depends on the homotopy classes of the maps $i$, $H$, and $f$.
\end{rem}

%% file: DetectingNullhomotopy.tex
\section{Null-homotopies from bordism of connective covers}\label{sec:detecting_null-homotopies}
In this section we record some basic facts from algebraic topology.

Suppose that $G$ is an infinite loop space that admits a map $G \to O$ of infinite loop spaces. Let $G\langle k\rangle$ denote the $(k-1)$-connected cover of $G$. We let $BG\langle k \rangle \coloneqq B(G\langle k\rangle)$. Examples include $G = O$, $U$, and $Sp$.

The first few connective covers in the case $G = O$ are
\[
SO = O\langle 1\rangle, \quad Spin = O\langle 2 \rangle = O\langle 3\rangle,\quad String = O\langle 4\rangle = \cdots = O\langle 7\rangle, \quad Fivebrane = O \langle 8\rangle.
\]

Note that if $k\geq 1$ and $n \geq 2k+2$, $S^n$ admits a tangential $G\langle k+1\rangle$-structure, i.e., a lift of the classifying map for the tangent bundle to a map $S^n \to BG\langle k+1\rangle$, since $\pi_n BG \cong \pi_n BG\langle k+1\rangle$. A choice of a tangential $G\langle k+1\rangle$-structure induces a choice of $MG\langle k+1\rangle$-fundamental class that we denote by $[S^n]$ below.
\begin{lem}\label{lem:DetectingNullHomotopy_gen}
    Let $k \geq 1$ and let $n \geq 2k+2$. Suppose that $Z^{2n}$ is a smooth manifold obtained from attaching handles of indices $0 \leq i \leq n-k$, such that there exists a retract to an $(n-k)$-dimensional cellular complex denoted by $\Core Z$.
    
    If $f\colon S^n \to \Core Z$ is a smooth map such that $f_\ast[S^n] = 0$ in $\widetilde H_n(\Core Z;MG\langle k+1\rangle)$, then the composition
	\[
		q \circ f \colon S^n \longrightarrow \Core Z/(\Core Z)_{n-k-1}
	\]
	is null-homotopic, where $q\colon \Core Z \to \Core Z/(\Core Z)_{n-k-1}$ denotes the quotient by the $(n-k-1)$-skeleton of $\Core Z$.
\end{lem}

The proof of \cref{lem:DetectingNullHomotopy_gen} requires the following two lemmas.

\begin{lem}\label{lem:HomotopyOfWedge}
Suppose that $M$ is an $m$-dimensional cellular complex with $m \leq 2n-2$. Suppose that there is a continuous map
\[ f\colon M \longrightarrow \bigvee_{i=1}^k S^n. \]
If the composition
\[ M \overset{f}{\longrightarrow}  \bigvee_{i=1}^k S^n \longrightarrow \prod_{i=1}^k S^n \]
is null-homotopic, then $f$ is null-homotopic.
\end{lem}

\begin{lem}\label{lma:fundamental_class}
    If $f \colon S^n \to S^{n-k}$ is a continuous map, the following diagram is commutative
    \[
    \begin{tikzcd}[sep=scriptsize]
        \widetilde{H}_n(S^n;MG\langle k+1\rangle) \rar{f_\ast} \dar{\cong} & \widetilde{H}_n(S^{n-k};MG\langle k+1\rangle) \dar{\cong} \\
        \pi_0(MG\langle k+1\rangle) \rar{f_\ast} & \pi_k(MG\langle k+1\rangle) \\
        \pi_0^{\mathrm{st}} \uar{\eta_\ast}[swap]{\cong} \rar{f_\ast} & \pi_k^{\mathrm{st}} \uar{\eta_\ast}[swap]{\cong}
    \end{tikzcd},
    \]
    where $\eta \colon \IS \to MG\langle k+1\rangle$ is the unit map.
\end{lem}
\begin{proof}
    Since the map $B1\to BH\langle k+1\rangle$ is $(k+1)$-connected, it implies that the unit map $\eta \colon \IS \to MG\langle k+1\rangle$ is $(k+1)$-connected, showing that the lower vertical maps $\eta_\ast$ are isomorphisms.
\end{proof}

We postpone the proof of \cref{lem:HomotopyOfWedge}. We now prove \cref{lem:DetectingNullHomotopy_gen}.

\begin{proof}[Proof of \cref{lem:DetectingNullHomotopy_gen}]
Note $\Core Z/(\Core Z)_{n-k-1} \simeq \bigvee_{i=1}^j S^{n-k}$ and let
\[ q_i \colon \bigvee_{i=1}^j S^{n-k} \longrightarrow S^{n-k}\]
denote the $i$-th projection. By \cref{lem:HomotopyOfWedge} it suffices to show that $q_i\circ q\circ f$ is null-homotopic for every $i\in \{1,\ldots,j\}$ since $n\geq2k+2$. By assumption we have $(q_i\circ q\circ f)_\ast[S^n] = 0$ in $\widetilde{H}_n(S^{n-k};MG\langle k+1\rangle)$. Since $[S^n]$ is a generator of $\widetilde{H}_n(S^{n-k};MG\langle k+1\rangle)$, it follows that the pushforward map
\[
(q_i\circ q\circ f)_\ast \colon \widetilde{H}_n(S^n;MG\langle k+1\rangle) \longrightarrow\widetilde{H}_n(S^{n-k};MG\langle k+1\rangle),
\]
is zero. By \cref{lma:fundamental_class}, it follows that the corresponding map $(q_i\circ q\circ f)_\ast \colon \pi_0^{\mathrm{st}} \to \pi_k^{\mathrm{st}}$ is zero. This implies that for $p \gg 0$, the map
\[
\varSigma^{p-n}(q_i\circ q\circ f) \colon S^p \longrightarrow \varSigma^{-k}S^{p}
\]
is null-homotopic. Since $n\geq 2k+2$, we have that $p = n$ is within the stable range, showing that $q_i\circ q\circ f$ is null-homotopic, thus finishing the proof.
\end{proof}

Now we turn our attention to proving \cref{lem:HomotopyOfWedge}. We will need the following elementary facts from algebraic topology.

\begin{lem}\label{lem:SomeElementaryATFacts}
Let $X$ and $Y$ be cellular complexes.
\begin{enumerate}
	\item There is a homotopy fiber sequence
	\[
		\varSigma(\varOmega X \wedge \varOmega Y) \longrightarrow X \vee Y \longrightarrow X \times Y.
	\]
	\item If $X$ is $a$-connected and $Y$ is $b$-connected, then $\varSigma (\varOmega X \wedge \varOmega Y)$ is $(a+b)$-connected.
\end{enumerate}
\end{lem}

\begin{proof}
Since $X$ and $Y$ are cellular complexes, the inclusions of their basepoints, say $x_0$ and $y_0$ respectively, are cofibrations. Consequently, the diagram
\[
	\begin{tikzcd}[row sep=scriptsize, column sep=scriptsize]
		\ast \rar \dar[swap] & Y \dar{i_Y} \\
		X \rar{i_X} & X \vee Y
	\end{tikzcd}
\]
is a homotopy pushout square. Taking the homotopy pullback of this diagram along the map $\phi \colon X \vee Y \to X \times Y$ yields a commutative diagram (up to homotopy)
\[
	\begin{tikzcd}[row sep=scriptsize, column sep=scriptsize]
		\varOmega X \times \varOmega Y \rar{\mathrm{pr}_{\varOmega Y}} \dar[swap]{\mathrm{pr}_{\varOmega X}} & \varOmega Y \dar \\
		\varOmega X \rar & F
	\end{tikzcd},
\]
where $F$ denotes the homotopy fiber of $\phi$.

By Mather's second cube theorem, see \cite[Theorem 25]{mather1976cubetheorems}, the homotopy pullback of a homotopy pushout square is a homotopy pushout square. Consequently, $F$ is the homotopy pushout of the above diagram, which is well-known to be
\[ \varOmega X \star \varOmega Y \simeq \varSigma (\varOmega X \wedge \varOmega Y),\]
where $\star$ denotes join. This proves that the sequence stated in the lemma is, in fact, a homotopy fiber sequence.

For the second item, we proceed as follows. First, $\varOmega X$ and $\varOmega Y$ are $(a-1)$-connected and $(b-1)$-connected respectively. By cellular approximation, $\varOmega X$ and $\varOmega Y$ are homotopy equivalent to spaces with no cells in dimension $1 \leq d \leq a-1$ and $1\leq d \leq b-1$, respectively. It follows that $\varOmega X \wedge \varOmega Y$ has no cells in dimension $1 \leq d \leq a+b-1$, and hence that $\varSigma(\varOmega X \wedge \varOmega Y) \simeq S^1 \wedge \varOmega X \wedge \varOmega Y$ has no cells in dimension $1\leq d \leq a+b$, finishing the proof.
\end{proof}

\begin{proof}[Proof of \cref{lem:HomotopyOfWedge}]
The proof proceeds by induction on $i$ for $\bigvee_{i=1}^k S^n$. Let
\[ \phi \colon \left(\bigvee_{i=1}^{k-1} S^n \right) \vee S^n \longrightarrow \left(\bigvee_{i=1}^{k-1} S^n \right) \times S^n \]
denote the natural inclusion. By the induction hypothesis, $\phi \circ f$ is null-homotopic. By \cref{lem:SomeElementaryATFacts}(i) and the homotopy lifting property for the homotopy fiber sequence, it suffices to show that
\[ \left[X, \varSigma \left( \varOmega  \left(\bigvee_{i=1}^{k-1} S^n \right) \wedge \varOmega S^n\right)\right]=0\]
in order to obtain that $f$ is null-homotopic. By the Hurewicz theorem, $\bigvee_{i=1}^{k-1} S^n$
and $S^n$ are both $(n-1)$-connected. So by the second item of \cref{lem:SomeElementaryATFacts}, 
\[\varSigma \left( \varOmega  \left(\bigvee_{i=1}^{k-1} S^n \right) \wedge \varOmega S^n\right)\]
is $(2n-2)$-connected. Since $X$ is a cellular complex whose dimension is less than or equal to $2n-2$, the cellular approximation theorem implies that
\[ \left[X, \varSigma \left( \varOmega  \left(\bigvee_{i=1}^{k-1} S^n \right) \wedge \varOmega S^n\right)\right]=0,\]
as desired.
\end{proof}

%% file: GeometricBackground.tex
\section{Handle presentation of Weinstein sectors}\label{sec:geom_background}
     The purpose of this section is to prove that the two definitions of a Weinstein sector in the literature (\cite{ganatra2020covariantly} and \cite[Definition 2.7]{chantraine2017geometric}) are equivalent. Recall from \cite{ganatra2020covariantly} that a Liouville sector is a Liouville manifold-with-boundary $(X, \lambda, Z)$ for which there is a function $I \colon \partial X \to \IR$ such that $I$ is \emph{linear at infinity}, meaning $ZI = I$ outside a compact set, where $Z$ denotes the Liouville vector field, and the Hamiltonian vector field $X_I$ of $I$ is outward pointing along $\partial X$.

    Following \cite{eliashberg2018weinstein,alvarez2025arborealization}, we now define Liouville pairs. A Liouville hypersurface in a Liouville $2n$-manifold $(X^{2n},\lambda)$ is a Liouville $(2n-2)$-manifold $(F,\lambda_F)$ together with a choice of Liouville embedding $(F, \lambda_F) \hookrightarrow (X \smallsetminus \Core X, \lambda)$ such that the induced map $F \to \partial X$ is an embedding. A Liouville pair is a tuple $(X,F)$ of a Liouville manifold and a Liouville hypersurface.

    A Weinstein domain is a Liouville domain $(X,\lambda)$ whose Liouville vector field admits a Lyapunov function that is constant along $\partial X$. Similarly a Weinstein manifold is a Liouville manifold whose Liouville vector field admits a Lyapunov function.

    \begin{defn}[Weinstein pair]
    	A \emph{Weinstein pair} is a Liouville pair $(X,F)$ such that both $X$ and $F$ are Weinstein manifolds, and the associated Liouville embedding $(F,\lambda_F) \hookrightarrow (X \smallsetminus \Core X, \lambda)$ preserves the Lyapunov functions.
    \end{defn}
    \begin{defn}[Weinstein sector {\cite{ganatra2022sectorial}}]\label{dfn:weinstein_sector}
    	A \emph{Weinstein sector} is a Liouville sector such that its convexification (see \cite[Section 2.7]{ganatra2020covariantly}) admits the structure of a Weinstein pair.
    \end{defn}
    A Liouville (Weinstein) cobordism is a Liouville (Weinstein) domain $X$ with $\partial X = \partial_- X \sqcup \partial_+ X$ such that the Liouville vector field points outwards along $\partial_+ X$ and inwards along $\partial_- X$.
    
    Any (finite type) Weinstein manifold can equivalently be defined as the result of successive \emph{Weinstein handle attachments} \cite{weinstein1991contact}. There is an alternative definition of (finite type) Weinstein sectors via Weinstein handles and Weinstein half-handles, see \cite[Definition 2.3]{chantraine2017geometric}. We review these definitions now.
    
    Consider $\IR^{2n}$ with the standard symplectic form $\omega = \sum_{j=1}^n dx_j \wedge dy_j$. For $0\leq k \leq n$, the vector field
    \[
        Z_k = \frac 12\sum_{j=1}^{n-k}(x_j \partial_{x_j} + y_j \partial_{y_j}) + \sum_{j=n-k+1}^n (2x_j \partial_{x_j} - y_j \partial_{y_j})
    \]
    is a Liouville vector field for $\omega$, corresponding to the Liouville one form
    \[
        \lambda_k = \frac 12 \sum_{j=1}^{n-k} (x_j dy_j - y_j dx_j) +\sum_{j=n-k+1}^n (2x_j dy_j + y_j dx_j).
    \]
    The vector field $Z_k$ for $0\leq k \leq n$ is gradient-like for the function
    \[
        f_k(x,y) = \frac 14\sum_{j=1}^{n-k} (x_j^2 + y_j^2) + \sum_{j=n-k+1}^n \left(x_j^2 - \frac 12 y_j^2\right).
    \]
    \begin{defn}[Weinstein handle]\label{dfn:weinstein_handle}
        Let $k\in \{0,\ldots,n\}$. The \emph{Weinstein handle of index $k$ and size $\delta > 0$} is the Weinstein cobordism $(H^\delta_k, Z_k, f_k)$, where
        \[
            H^\delta_k \coloneqq \left\{(x,y) \in \IR^{2n} \mid -\delta \leq f_k(x,y) \leq \delta \right\},
        \]
        see \cref{fig:handle}. A Weinstein handle of index equal to $n$ is called \emph{critical}. A Weinstein handle that is not critical is called \emph{subcritical}.
    \end{defn}
    \begin{figure}[!htb]
        \centering
        \includegraphics{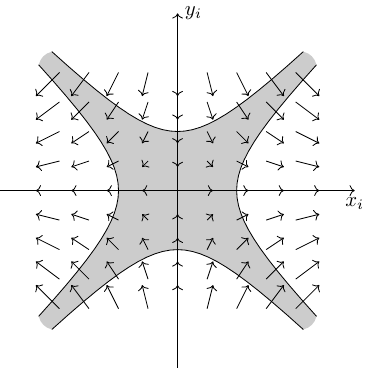}
        \caption{The projection of the Weinstein handle $H^\delta_k$ in the $(x_i,y_i)$-plane together with its Liouville vector field.}\label{fig:handle}
    \end{figure}
    The positive and negative contact boundaries of $H^\delta_k$ are defined as
    \[
        \partial_{\pm} H^\delta_k = \left\{(x,y) \in \IR^{2n} \mid f_k(x,y) = \pm \delta\right\},
    \]
    respectively. Topologically we have $H^\delta_k \cong D^{2n-k} \times D^k$, where the second factor corresponds to the last $k$ coordinates $y_{n-k+1},\ldots,y_n$ of $\IR^{2n}$. The disk $(\left\{0\right\} \times \IR^k) \cap H^\delta_k \subset H^\delta_k$ is the \emph{core disk} of the handle with \emph{attaching sphere} $(\left\{0\right\} \times \IR^k)\cap \partial_{-} H^\delta_k \subset \partial_{-} H^\delta_k$.
    Dually, the disk $(\IR^{2n-k} \times \left\{0\right\}) \cap H^\delta_k\subset H^\delta_k$ is called the \emph{cocore disk} with \emph{cocore sphere} $(\IR^{2n-k} \times \left\{0\right\}) \cap \partial_+ H^\delta_k \subset \partial_+ H^\delta_k$.
    \begin{defn}[Weinstein half-handle]\label{dfn:weinstein_half-handle}
        Let $k\in \{0,\ldots,n\}$. The \emph{Weinstein half-handle of index $k$ and size $\delta > 0$} is defined as the Weinstein sector $(\varPi^\delta_k, Z_k, I_k, f_k)$, where
        \[
            \varPi^\delta_k \coloneqq \left\{(x,y) \in \IR^{2n} \mid -\delta \leq f_k(x,y) \leq \delta, \; y_n \geq 0\right\},
        \]
        see \cref{fig:half-handle}.

        A Weinstein half-handle of index equal to $n$ is called \emph{critical}. A Weinstein half-handle that is not critical is called \emph{subcritical}.
    \end{defn}
    Topologically we have $\varPi^\delta_k \cong D^{2n-k} \times D^k_+$, where the second factor corresponds to the last $k$ coordinates $y_{n-k+1},\ldots, y_n$ of $\IR^{2n}$, with the constraint $y_n \geq 0$. The (horizontal) boundary of the half-handle is
    \[
        \partial \varPi^\delta_k \coloneqq \left\{(x,y) \in \varPi^\delta_k \mid y_n = 0\right\}.
    \]
    The function $I_k$ is defined as $I_k \colon \partial \varPi^\delta_k \to \IR$, $(x,y) \mapsto -\frac{x_n}{2}$. We note that the symplectic boundary of a Weinstein half-handle of index $k$ is in fact a Weinstein handle of index $k-1$. Namely it is given by $(F_k, Z_{F_k}, f_{F_k})$ where $F_k \coloneqq I_k^{-1}(0) \cong H^\delta_{k-1}$, $Z_{F_k} \coloneqq Z_k|_{F_k}$ and $f_{F_k} \coloneqq f_k|_{F_k}$.
    \begin{figure}[!htb]
        \centering
        \includegraphics{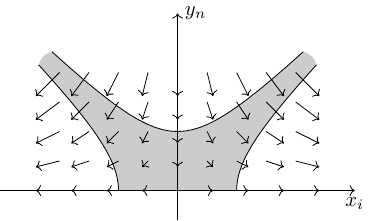}
        \caption{The projection of the Weinstein half-handle $\varPi^\delta_k$ in the $(x_i,y_n)$-plane together with its Liouville vector field.}\label{fig:half-handle}
    \end{figure}
    Similar to the case with standard Weinstein handles above, the half-disk $(\left\{0\right\} \times \IR^k_+) \cap \varPi^\delta_k \subset \varPi^\delta_k$ is called the \emph{core disk} of the half handle. We define the \emph{cocore disk} of $\varPi^\delta_k$ to be the spreading of the cocore disk of its symplectic boundary $H^\delta_{k-1}$ (see \cite[Definition 2.11]{chantraine2017geometric} for the definition of spreading).
    \begin{rem}\label{rem:cocore_linking_disk}
        Under convexification (see \cite[Section 2.7]{ganatra2020covariantly}), the Lagrangian cocore disks of critical Weinstein half-handles become the linking disks, see \cite[Section 5.3]{ganatra2022sectorial}.
    \end{rem}
    \begin{defn}[Sectorial Morse function {\cite[Definition 2.7]{chantraine2017geometric}}]\label{dfn:sectorial_morse_fcn}
        Let $X$ be a Liouville sector. A \emph{sectorial Morse function} is an exhausting Morse function $f \colon X \to \IR$ that has a finite number of critical points, is a Lyapunov function for the Liouville vector field of $X$, and moreover satisfies the following:
        \begin{enumerate}
            \item $df(C) > 0$ on $\left\{I > 0\right\}$ and $df(C) < 0$ on $\left\{I < 0\right\}$, where $C$ is the characteristic foliation of $\partial X$ which is oriented so that $\omega(C,N) > 0$ where $N$ is an outward pointing normal vector field along $\partial X$.
            \item The Hessian of $f$ at a critical point $p\in \partial X$ evaluates negatively on $N$.
            \item There is a constant $c\in \IR$ whose sublevel set satisfies $\left\{f \leq c\right\} \subset X \smallsetminus \partial X$ and contains all interior critical points of $f$.
        \end{enumerate}
    \end{defn}
    We now record the fact that the two definitions of a Weinstein sector in the literature (\cite{ganatra2020covariantly} and \cite[Definition 2.7]{chantraine2017geometric}) are equivalent.
    \begin{lem}\label{lem:equiv_of_defns}
        Let $X$ be a Liouville sector. The existence of a sectorial Morse function on $X$ is equivalent to $X$ being a Weinstein sector in the sense of Ganatra--Pardon--Shende \cite{ganatra2022sectorial}.
    \end{lem}
    \begin{proof}
        If $X$ is a Weinstein sector in the sense of \cref{dfn:weinstein_sector}, then we consider its convexification which is the Weinstein pair $(\overline X, F)$. Attach $(F \times T^\ast [-1,0], \lambda_F + 2ydx+xdy, f_F + y^2-\frac 12x^2)$ to $\partial \overline X$ along a small neighborhood of $F \subset \partial \overline X$ (cf.\@ \cite[Section 2.2]{asplund2022chekanov} and \cite[Section 2.1]{asplund2021simplicial}). The result is $X$ (up to deformation) together with a Morse function $f$ which can easily be checked to be sectorial.

        Conversely, if $X$ is a Liouville sector with a sectorial Morse function $f$, we construct the convexification of $X$ by attaching $(F \times T^\ast[0,1], \lambda_F + \frac 12(ydx-xdy))$ to $X$ along $\partial X$. We equip $F \times T^\ast [0,1]$ with the Morse function $\overline f \coloneqq f_0 + \frac 14 x^2$, where $f_0 \coloneqq f|_{\partial X}$. This indeed makes $F \times T^\ast [0,1]$ Weinstein and hence $(\overline X, F)$ a Weinstein pair. The Liouville vector field on $F \times T^\ast[0,1]$ is given by $Z = Z_0 + \frac 12 x \partial_x$ where $Z_0 \coloneqq (Z_X)|_{\partial X}$, and 
        \begin{align*}
            d \overline f(Z) &= df_0(Z_0) + \frac 14 x^2 \geq \delta_0(|df_0|^2 + |Z_0|^2) + \frac 14 x^2 \\
            &\geq \delta \left(|df_0|^2 + \frac 14 x^2 + |Z_0|^2 + \frac 14x^2\right) = \delta(|d \overline f|^2 + |Z|^2),
        \end{align*}
        where $\delta \coloneqq \min \left(\delta_0, \frac 12\right)$.
    \end{proof}
    A consequence of \cref{lem:equiv_of_defns} is that any Weinstein sector in the sense of \cref{dfn:weinstein_sector} is Weinstein isomorphic (up to deformation) to a Weinstein handlebody with boundary, which by definition is the result of consecutive attachments of Weinstein handles and Weinstein half-handles \cite[Section 2.3]{chantraine2017geometric}.
   
    \begin{defn}[Subcritical Weinstein sector]\label{dfn:subcrit_weinstein_sector}
    	We say that a $2n$-dimensional Weinstein sector is \emph{subcritical} if it is the result of successive attachments of Weinstein handles and Weinstein half-handles of index $< n$.
    \end{defn}
    \begin{defn}[Handlebody decomposition of a Weinstein manifold]\label{dfn:handlebody_decomp}
        A \emph{handlebody decomposition} of a Weinstein manifold $X$ is a pair $(X_0,\varLambda)$ consisting of a subcritical Weinstein manifold $X_0$ and a Legendrian submanifold $\varLambda \subset \partial_\infty X_0$ such that the attachment of a critical Weinstein handle to $X_0$ along $\varLambda$ recovers $X$.
    \end{defn}
    \begin{defn}[Handlebody decomposition of a Weinstein pair]
        A \emph{handlebody decomposition} of a Weinstein pair $(X,F)$ is a pair $((X_0,\varLambda_X),(F_0,\varLambda_F))$ of a handlebody decomposition of $X$ and a handlebody decomposition of $F$.
    \end{defn}
    \begin{defn}[Handlebody decomposition of a Weinstein sector]\label{dfn:handlebody_decomp_sector}
        A handlebody decomposition of a Weinstein sector is defined as a handlebody decomposition of its corresponding Weinstein pair via convexification (see \cite[Section 2.7]{ganatra2020covariantly}).
    \end{defn}
    \begin{exmp}
        As a simple example, the (open) Weinstein sector $T^\ast \IR^n$ is equivalent to the Weinstein pair $(B^{2n},N(\varLambda_0))$, where $\varLambda_0 \subset S^{2n-1} = \partial_\infty B^{2n}$ is the standard $(n-1)$-dimensional Legendrian unknot, and $N(\varLambda_0) \cong T^\ast S^{n-1}$ denotes a Weinstein neighborhood of it. Any Weinstein handlebody decomposition of $T^\ast \IR^n$ is equivalent by definition to a Weinstein handlebody decomposition of the Weinstein pair $(B^{2n},N(\varLambda_0))$. The simplest Weinstein handlebody decomposition of this Weinstein pair is the pair $((B^{2n},\varnothing),(B^{2n-2},\varLambda'_0))$, where $\varLambda'_0 \subset S^{2n-3} = \partial_\infty B^{2n-2}$ is the standard $(n-2)$-dimensional Legendrian unknot. The original open Weinstein sector is the completion of the Weinstein sector $T^\ast B^n$, and the Weinstein handlebody decomposition just described corresponds to the choice of a sectorial Morse function on $B^n$ with one index $0$ critical point in the interior, and two critical points at the boundary, of indices $1$ and $n$, respectively (where the index is calculated in the whole ball). In particular we remark that the Morse function with a single interior critical point of index $0$ is not a sectorial Morse function in the sense of \cref{dfn:sectorial_morse_fcn}.
    \end{exmp}